\documentclass{amsart}

\usepackage{graphicx}
\usepackage{amsmath}
\usepackage{amssymb}
\usepackage{bbm}
\def\RR{\mathbb R}

\newcommand{\tr}{\hbox{{\rm tr}}}
\newcommand{\diag}[1]{\hbox{{\rm diag}}( #1 )}

\def\R#1{$(\ref{#1})$}
\def\D{\,{\rm d}}

\newcommand{\de}{\partial}
\newcommand{\sinc}{\operatorname{sinc}}

\newcommand{\one}{\mathbbm{1}}
\newcommand{\LL}{\mathcal{L}}
\newtheorem{theorem}{Theorem}[section]
\newtheorem{coroll}{Corollary}[theorem]
\newtheorem{lemma}[theorem]{Lemma}

\newtheorem{definition}{Definition}[section]

\newenvironment{meqn}
{\arraycolsep=1.4pt
  
  \begin{array}{rcl}}
  {\end{array}}

\begin{document}
\author[A. Zanna]{Antonella Zanna$^\dag$}
\thanks{$^\dag$ Matematisk institutt, Universitetet i Bergen, Norway,
  email: \texttt{Antonella.Zanna@uib.no}}
\date{\today}
\title{Symplectic P-stable Additive Runge--Kutta Methods}

\begin{abstract}
  Symplectic partitioned Runge--Kutta methods can be obtained from a
  variational formulation where all the terms in the discrete Lagrangian are
  treated with the same quadrature formula. We construct a family of
  symplectic methods allowing the use of different quadrature formulas
  (primary and secondary) for different terms of the Lagrangian. In particular, we study a
  family of methods using Lobatto quadrature (with corresponding
  Lobatto IIIA-B symplectic pair) as a primary method and
  Gauss--Legendre quadrature as a secondary method.
  The methods have the same favourable implicitness as the underlying Lobatto
   IIIA-B pair, and, in addition, they are \emph{P-stable}, therefore
   suitable for application to highly  oscillatory problems.
\end{abstract}

\maketitle

\section{Introduction}
\label{sec:introduction}
In this paper we introduce a family of Runge--Kutta methods
of additive type particularly suited to highly oscillatory problems.
Our method are derived from a variational
formulation, using different quadrature formulas for different parts
of the Lagrangian. We will consider mainly a formulation where we use
a primary method (giving rise to a symplectic PRK) and a secondary
method, that is based on different quadrature weight and nodes. 
The final formulation of the hybrid method can be classified as
special subclass of symplectic 
Additive Runge--Kutta (ARK) methods.
ARK were introduced already in the 80's \cite{cooper83ark} to deal
with stiff ODEs. These methods have been recently generalized by \cite{sandu15gark}
(GARK methods) to add flexibility in treating different force terms in
the differential equation by different sets of coefficients. In the context of
algebraic differential equations, similar approaches have been
followed by \cite{jay98spark} with special attention to structure
preservation in Hamiltonian systems. Recently, GARK methods for stiff
ODEs and DAEs were considered in \cite{tanner18gark} with focus
on the combination Gauss/Radau IIA and Gauss/Lobatto IIIC.

Our motivation comes from the study of highly oscillatory problems, 
trigonometric integrators (see \cite{hairer06gni} and references therein),
in particular, the intriguing properties of the second-order
implicit-explicit (IMEX) method originally proposed by
\cite{skeel97cis} and further analyzed in a variational setting in \cite{stern09imex}
and as a modified trigonometric integrator in
\cite{mclachlan14mti}. This method is equivalent to applying the
``midpoint rule''\footnote{In facts, the method uses a linear
  interpolation for the internal stage of the implicit midpoint rule.}
to the fast, linear part of the 
system, and the leapfrog (St{\"o}rmer/Verlet) method to the slow,
nonlinear part. It has the following properties: (i) it
is symplectic; (ii) it is free of artificial resonances; (iii) it is
the unique method that correctly captures slow energy exchange to
leading order; (iv) it conserves the total energy and a modified
oscillatory energy up to to second order; (v) it is uniformly
second-order accurate in the slow components; and (vi) it has the
correct magnitude of deviations of the fast oscillatory energy, which
is an adiabatic invariant \cite{mclachlan14mti}.

The St{\"o}rmer/Verlet method belongs to the family of Lobatto IIIA-B
partitioned Runge--Kutta methods (PRK). In an unpublished report from 1995, Jay and Petzold studied the
linear stability of Lobatto PRK and proved that none of the methods in
this family is P-stable, as they are not unconditionally stable when
applied to the harmonic oscillatory. They concluded that these methods
were not suitable for highly oscillatory systems
\cite{jay95hos}. Further stability properties of these Lobatto PRK
were also studied in the context of multisymplectic integration and the
wave equation in \cite{mclachlan11lso}.

Being the lack of P-stability well established for Lobatto PRK, it is
therefore quite a surprise that the combination implicit midpoint and
St{\"o}rmer/Verlet is unconditionally stable. Intrigued by the
properties of the IMEX, \cite{zanna17afo} introduced a family of
symplectic, unconditionally stable modified trigonometric integrators
of second order, which included the IMEX as a special case. 

In this paper we construct higher order integrators  pursuing the
variational approach of the Lagrangian formalism.
The technique used is very close to the one described in
\cite{hairer06gni} for the derivation of symplectic PRK methods.
The main idea is similar to that described
above for the second order IMEX: to use a Lobatto method for the
kinetic energy and slow potential (the latter being costly to compute)
and Gauss-Legendre of the same order for the linear highly oscillatory
part (easy to compute). To avoid the introduction of further function evaluation of the
potential, we approximate the internal stages values by two
techniques, interpolation and collocation. Although we focus
especially on the Lobatto and Gauss--Legendre combination as primary
and secondary method respectively, the derivation presented is general
and applies to different combinations of primary and secondary
methods.

Other variational approaches exist, especially in the  community of
computational mechanics, see for instance
\cite{marsden_west_2001}. Recently, the latter approach has been used,
together to a splitting of the Lagrangian, in the context of higher
order variational integrators for dynamical systems with holonomic
constraints \cite{wenger2017caa} and in order to devise mixed order
integrators for systems with multiple scales \cite{wenger2016vio}.
The approach in \cite{wenger2017caa} and the one presented in this
paper have several common features but also diversities, like the
choice of the independent variables with respect to which the
variations are done. Having said this, it is not unlikely that some of the
methods derived by the two approaches will coincide for some similar
choices of coefficients and some problems, but a thorough comparison
is outside the scope of the present paper.

The paper is organized as follows. In Section~\ref{sec:vari-deriv} we
show the general theory for the derivation of the methods and how to
construct the coefficients by either interpolation or collocation. In
Section~\ref{sec:order} we prove some results on the order of the
proposed methods.
In Section~\ref{sec:P-stability} we study the P-stability
of the methods and
in Section~\ref{sec:methods-as-modified} we show how the methods can be put in
the framework of modified trigonometric integrators.
In Section~\ref{sec:numer-exper} we show several numerical
tests on the Fermi-Pasta-Ulam-Tsingou problem and compare with higher order
construction of the IMEX method using the Yoshida time-stepping
technique. Finally, we have some concluding remarks and in the
Appendix we present explicitly tables with the coefficients for the methods of
the Lobatto--Gauss-Legendre family based on interpolation for order
two, four and six.

\section{Variational derivation}
\label{sec:vari-deriv}
It is well known that symplectic Partitioned Runge-Kutta methods (PRK)
can be obtained by a variational method, doing discrete variations on
a discrete Lagrangian approximating the continuous Lagrangian
$L(q,\dot q)$ \cite{hairer06gni}.

Consider a Lagrangian $L(q,\dot q)$ and assume that it can be written
as sum of two (or more) terms,
\begin{displaymath}
  L(q,\dot q) = L^1(q, \dot q) + L^2 (q, \dot q)+L^3(q, \dot q)+ \cdots.
\end{displaymath}
Whereas the derivation of symplectic PRK uses the same quadrature for all the
terms, we consider the case when one would like to use a different
quadrature for one or more terms in the sum.
A motivating example is the case of highly oscillatory problems in
molecular dynamics, with a Lagrangian of the form
\begin{displaymath}
  L(q,\dot q) = T(\dot q) - V^1(q) -V^2(q),
\end{displaymath}
where $V^1$ is a slow potential while $V^2$ is a fast oscillating
potential, for instance of the form $V^2 = -\frac12 q^T \Omega^2 q$,
$\Omega$ being a diagonal matrix with elements $\omega_i \gg 1$. A natural splitting in this context would be 
\begin{displaymath}
  L^1 = T-V^1 , \qquad L^2 = - V^2.
\end{displaymath}
In this paper, we restrict the discussion to the case when the
Lagrangian is split in two terms as above, but the generalization to
several terms is straightforward.

We focus on a discrete Lagrangian of the form
\begin{equation}
  \label{eq:1}
  L_h = h \sum_{i=1}^{s_1} b_i L^1(Q_i, \dot Q_i) + h \sum_{k=1}^{s_2} \tilde b_k L^2 (\tilde Q_k),
\end{equation}
with
\begin{eqnarray}
  \label{eq:2}
  Q_i &=& q_0 + h \sum_{j=1}^{s_1} a_{i,j} \dot Q_j\\
  q_1 &=& q_0 + h \sum_{i=1}^{s_1} b_i \dot Q_i
          \label{eq:3}
\end{eqnarray}
where the coefficients $(A, b,c)$ are the coefficients of a standard
RK method with $s_1$ stages (\emph{primary} method), while $(\tilde b,
\tilde c)$ are the weights and nodes of the \emph{secondary}
quadrature with $s_2$ weights and nodes respectively.
To avoid the introduction of extra internal stages due to the
secondary method, we assume that the $\tilde Q_i$ can be written as
\begin{equation}
  \label{eq:4}
  \tilde Q_i = q_0 + h \sum_{j=1}^{s_1} \tilde a_{i,j} \dot Q_j
\end{equation}
for some coefficients $\tilde a_{i,j}$, with $i=1, \ldots, s_1$ and
$j=1, \ldots, s_2$ to be determined. 
Because of the linear dependence between the $Q_i$s, the $\tilde{Q}_i$ and $\dot Q_i$s,
we perform the variation of \R{eq:1} using the method of Lagrange
multipliers in a manner very similar to the derivation of symplectic
PRK described in \cite{hairer06gni}. The augmented discrete Lagrangian
using the constraint \R{eq:3} is then
\begin{equation}
  \label{eq:5}
  h \sum_{i=1}^{s_1} b_i L^1(Q_i, \dot Q_i) + h \sum_{k=1}^{s_2} \tilde b_k L^2 (\tilde Q_k) - \lambda (q_1 - q_0 - h \sum_{i=1}^{s_1} b_i \dot Q_i).
\end{equation}
The variation variables are now the $\dot Q_i$ and $\lambda$. Derivation with respect to $\lambda$ imposes the constraint \R{eq:3}, while derivation with respect to the $\dot Q_j$ gives the relation between the multiplier $\lambda$ and the other variables,
\begin{equation}
  \label{eq:6}
  \sum_{i=1}^{s_1} b_i\left( \frac{\de L^1(Q_i, \dot Q_i)}{\de q} \frac{\de Q_i} {\de \dot Q_j} \right) + b_j \frac{\de L^1}{\de \dot Q_j}+ \sum_{k=1}^{s_2} \tilde b_k \frac{\de L^2}{\de q} (\tilde Q_k) \frac{\de \tilde Q_k}{\de \dot Q_j }  = \lambda b_j
\end{equation}
We set
\begin{align}
  \label{eq:7}
  &P_j = \frac{\de L^1}{\de \dot q} (Q_j, \dot Q_j),  \quad \dot P_j
  =\frac{\de L^1}{\de  q} (Q_j, \dot Q_j), \\
  &\tilde  P_j = \frac{\de L^2}{\de \dot q} (\tilde Q_j)=0, \quad \dot{\tilde
  P}_j = \frac{\de L^2}{\de q} (\tilde Q_j) .
  \label{eq:8}
\end{align}
With this notation, and using the relations $\frac{\de Q_i}{\de \dot
  Q_j} = h a_{ij} I$, $\frac{\de \tilde Q_i}{\de \dot
  Q_j} = h \tilde a_{ij} I$, the constraint conditions in equation \R{eq:6} read
\begin{equation}
  \label{eq:9}
  b_j P_j = b_j \lambda - h \sum_{i=1}^{s_1}b_i a_{i,j} \dot P_i- h
  \sum_{k=1}^{s_2} \tilde b_k \tilde a_{k,j}\dot {\tilde P}_k.
\end{equation}
The symplectic method is obtained via the discrete Euler--Lagrange
equations, which can be formulated introducing the conjugate
variables $p_0$ and  $p_1$ as
\begin{equation}
  \label{eq:10}
  p_0 = -\frac{\de L_h}{\de q_0}, \qquad p_1 = \frac{\de L_h}{\de q_1}
\end{equation}
and thereafter eliminating $\lambda$ using \R{eq:9}.

By direct computation, we have
\begin{eqnarray}
  p_0 &=& -\frac{\de L_h}{\de q_0} = -h \sum_{i=1}^{s_1} b_i \dot P_i
          (I + h \sum_{l=1}^{s_1} a_{i,l}\frac{\de \dot Q_l}{\de q_0})
          \nonumber \\
      && \qquad \mbox{} - h \sum_{i=1}^{s_1} b_i P_i \frac{\de \dot
         Q_i}{\de q_0} - h \sum_{k=1}^{s_2} \tilde b_k \dot{\tilde
         P}_k (I + h \sum_{m=1}^{s_1} \frac{\de \dot Q_m}{\de q_0})
         \label{eq:11}
  \\
  &=& - h \sum_{i+1}^{s_1} b_i \dot P_i - h \sum_{k=1}^{s_2}
      \tilde b_k \dot{\tilde P}_k - \sum_{i=1}^{s_1} b_i P_i \frac{\de \dot
      Q_i}{\de q_0} + \sum_{l=1}^{s_1} (b_l P_l -\lambda b_l)
      \frac{\de \dot Q_l}{\de q_0}
      \label{eq:12}
  \\
  &=& - h \sum_{i+1}^{s_1} b_i \dot P_i - h \sum_{k=1}^{s_2}
      \tilde b_k \dot{\tilde P}_k + \lambda,
      \label{eq:13}
\end{eqnarray}
where in \R{eq:12} we have used \R{eq:9} and in \R{eq:13} we have used
$\sum_{i=1}^{s_1} b_l \frac{\de \dot Q_l}{\de q_0} = - I$ which comes
from the derivation of \R{eq:3}.

By a similar token, we find
\begin{equation}
  \label{eq:14}
  p_1 = \frac{\de L_h}{\de q_1}= \lambda,
\end{equation}
so that, eliminating $\lambda$, we get the relation
\begin{equation}
  \label{eq:15}
  p_1 = p_0 + h \sum_{i=1}^{s_1} b_i \dot P_i + h \sum_{k=1}^{s_2}
  \tilde b_k \dot {\tilde P}_k.
\end{equation}
To obtain the definition of the $P_j$, we use \R{eq:13} and
substitute in \R{eq:9} to obtain
\begin{equation}
  \label{eq:16}
  P_j = p_0 + h \sum_{i=1}^{s_1} \widehat
{a}_{i,j} \dot
  P_i +  h \sum_{k=1}^{s_2} (\tilde b_k -\frac{\tilde b_k \tilde
    a_{k,j}}{b_j} )\dot {\tilde P}_k.
\end{equation}
We recognize that the first set of coefficients is $\widehat
{a}_{i,j}= b_j - b_j a_{j,i}/b_i$, so that the $L^1$ part is treated
with a classical symplectic pair of PRK \cite{hairer06gni}.   The second set of coefficients
imposes the symplecticity condition for the use of the secondary
method in the treatment of the $L^2$.

\subsection{General format of the methods}
\label{sec:constr-meth-some}
The generalization to $L=L^1+L^2 + \cdots + L^n$ is
straightforward and leads to a symplectic subclass of ARK methods.
In this paper we describe in detail the case
$n=2$,  that is $L^1 = \frac12 \dot q ^T \dot q -
V^1(q)$, $L^2 = -V^2(q)$. Let $-\nabla V^1= F^1$ and $-\nabla V^2 = F^2$ the
forces corresponding to the potentials $V^1, V^2$.
Denote by $(A,b,c)$ the primary method so that, with $(\widehat A,
b,c)$, it forms symplectic PRK pair.  Let $(\tilde b,
\tilde c)$ be the secondary method (only quadrature weights and nodes
are necessary).
We have
\begin{displaymath}
  \frac{\de L^1}{\de
  q} = -\nabla V(q) = F^1(q),  \qquad  \frac{\de L^2}{\de
  q} = -\nabla V(q) = F^2(q) \qquad \frac{\de L^1}{\de \dot q} = \dot
q.
\end{displaymath}
The constraint relation \R{eq:6} 
\begin{displaymath}
  h \sum_{i=1}^{s_1} b_i a_{i,j} F^1(Q_i) + b_j \dot Q_j + h
  \sum_{k=1}^{s_2} \tilde b_k \tilde a_{k,j} F^2 (\tilde Q_k) =
  \lambda_j b_j
\end{displaymath}
allows us to find the derivatives $\dot Q_i(=P_i)$ at the intermediate stages.
The full method reads
\begin{equation}
  \label{eq:20}
  \begin{meqn}\displaystyle 
  p_1 &=& \displaystyle  
  p_0 + h \sum_{i=1}^{s_1} b_i F^1(Q_i) + h \sum_{k=1}^{s_2}
  \tilde b_k F^2(\tilde Q_k) \\
  P_i &=&\displaystyle  
  p_0  + h \sum_{j=1}^{s_1} \widehat a_{i,j} F^1 (Q_j) +
               h \sum_{k=1}^{s_2} \widehat{\tilde a_{i,k}}
               F^2(\tilde Q_k) \\
  Q_i &=& \displaystyle  
  q_0 + h \sum_{j=1}^{s_1} a_{i,j} P_j \\
  \tilde Q_i &=& \displaystyle  
  q_0 + h \sum_{j=1}^{s_2} \tilde a_{i,j} P_j,
\end{meqn}
\end{equation}
where $\widehat {\tilde a}_{i,k} = \tilde b_k - \frac{\tilde
  b_k \tilde a_{k,i}}{b_i}$. In matrix notation,
\begin{align}
  &\widehat{\tilde A} = (\one_{s_1\times s_2} - B^{-1} \tilde A^T)
    \tilde B, \qquad B = \diag{b}, \tilde B = \diag{\tilde b}, \label{eq:widehattildeA}\\
  & \widehat{A} = (\one_{s_1\times s_2} - B^{-1} A^T)B \label{eq:widehatA}
  \end{align}
The method requires the evaluation of extra internal stages for the
secondary method (the
$\tilde Q_i$) \emph{but only as many function evaluations of $F^1$ as
for the underlying (symplectic) PRK method}, allowing for different number of
function evaluations for $F^2$. This can be particularly interesting
when $F^1$ is expensive, while $F^2$ is cheap to compute.

The methods \R{eq:20}
can be applied to an Hamiltonian system
\begin{displaymath}
  \dot q = \frac{\de H}{\de p}, \qquad \dot p = -\frac{\de H}{\de q}.
\end{displaymath}
with Hamiltonian energy $H(q,p) = p^T \dot q - L(q, \dot q)$ where  $L(q, \dot q) =
\frac12 \dot q^T \dot q - V^1(q) -V^2(q)$.

\begin{theorem}
  The methods \R{eq:20} with $\widehat{\tilde A} $ and $\widehat A$ as
  in \R{eq:widehattildeA} and \R{eq:widehatA} are symplectic.
\end{theorem}
\begin{proof}
  Follows immediately from the fact that the variational derivation of
  the methods uses essentially generating forms of the first kind.
\end{proof}

While the symplectic requirements for the matrix $\widehat A$ are well
known in the the context of PRK, those for $\widehat{\tilde A}$ are
the same as those derived by algebraic arguments in
\cite{sandu15gark,jay98spark}. In their approaches, one uses the
relations \R{eq:widehattildeA} and \R{eq:widehatA} as (nonlinear) constraints to
solve for the coefficients of the methods, obeying, in addition, some
given order conditions. The solution of the nonlinear system needs not be unique.

In this paper, we follow a different approach which leads to at least two
solutions for the matrix ${\tilde A}$ and consequently
$\widehat{\tilde A}$.  

\subsection{Construction of the matrix $\tilde A$}
\label{sec:choice-interpolation}
There are two natural choices to construct the approximations $\tilde
Q_j$ in \R{eq:4}. We assume that the primary method is desribed
by the RK tableau
\begin{displaymath}
  \begin{array}{c| ccc}
    c_1 & a_{1,1}&\cdots &a_{1,s_1} \\
    \vdots & \\
    c_{s_1} & a_{s_1,1} & \cdots & a_{s_1,s_1}\\ \hline
        &b_1 & \cdots & b_{s_1}
  \end{array}.
\end{displaymath}
For the secondary method, we construct a tableau of the type
\begin{displaymath}
  \begin{array}{c| ccc}
    \tilde c_1 & \tilde a_{1,1}&\cdots &\tilde a_{1,s_1} \\
    \vdots & \\
    \tilde c_{s_2} & \tilde a_{s_2,1} & \cdots & \tilde a_{s_2,s_1}\\ \hline\\*[-10pt]
        &\tilde b_1 & \cdots & \tilde b_{s_2}
  \end{array}
\end{displaymath}
where the $\tilde c_k$ and $\tilde b_k$ are respectively the nodes and
weights of the secondary. Note that the matrix $\tilde A$ has
dimension $s_2\times s_1$.

\begin{description}
\item [Interpolation]   Given the primary nodes $c_1, \ldots c_{s_1}$, we let
  $\LL_i(t)=\prod_{k\not=i}\frac{t-c_k}{c_i-c_k}$ be the $i$th cardinal
  Lagrange polynomial and construct the  interpolating polynomial
  \begin{equation}
    \label{eq:17}
    \tilde Q(t) = \sum_{l=1}^{s_1} \LL_l(\tau) Q_l,
  \end{equation}
  where the $Q_l= q_0 + h \sum_{j=1}^{s_1} a_{l,j}\dot Q_j$ are
  obtained by the primary method. Substituting $Q_l$ in
  \R{eq:17}, recalling that $\sum_{j} \LL_j(\tau) = 1$ and computing
  in the nodes $\tilde c_i$ of the secondary method, we recover \R{eq:4}, with coefficients
  \begin{displaymath}
    \tilde a_{i,j} = \sum_{l=1}^{s_1} \LL_l(\tilde c_i) a_{l,j}, \qquad
    i = 1, \ldots, s_2, \quad j = 1, \ldots, s_1.
  \end{displaymath}
  Let $\LL(\tilde c)$ the $s_2\times s_1$ matrix with elements $\LL(\tilde c)_{i,j} =
  \LL_{j}(\tilde c_i)$ of the \emph{primary} Lagrange cardinal
  polynomials evaluated in the \emph{secondary} nodes. Then
  \begin{equation}
    \tilde A = \LL(\tilde c) A.
     \label{eq:18}
  \end{equation}
 where $A$ is the coefficient matrix of the primary method.

\item [Collocation]
  Another natural choice is to use the interpolation of the $\dot
  Q_j$s:  we use the cardinal interpolating polynomials $\LL_j(t)$
  constructed with the nodes of the   \emph{primary method} to
  construct $\dot  Q \approx \sum_{j=1}^{s_1}\LL_{j}(t) \dot Q_j$. The
  polynomial is integrated to obtain $\tilde
  Q(t)\approx  q_0 + h \int_{0}^t \sum_{j=1}^{s_1}\LL_{j}(\tau) \dot
  Q_j \D \tau$. Thereafter, evaluating in the \emph{secondary nodes} $\tilde c_i$, we recover
  \R{eq:4} with coefficients 
  \begin{equation}
    \label{eq:19}
    \tilde a_{i,j} = \int_0^{\tilde c_i} \LL_j(\tau) \D \tau , \qquad
    i = 1, \ldots, s_2, \quad j = 1, \ldots, s_1.
  \end{equation}
\end{description}

\section{Order of the methods}
\label{sec:order}
A general treatment of the order conditions for these ARK methods can
be developed using the algebraic tree theory in a manner very similar
to the order analysis of the ARK, GARK  methods \cite{sandu15gark,
  tanner18gark} using the formalism of colored trees
\cite{hairer06gni}. The order conditions are used to derive the
coefficients of the methods.

In our setting, the primary method, leading to a PRK pair $(A,
\widehat A, b,c)$, and the secondary method,  $(\tilde A,
\widehat{\tilde A}, \tilde b, \tilde c)$, are given by the choices
\R{eq:18}-\R{eq:19}, but the order of the resulting method \R{eq:20} is not obvious.

\begin{lemma}
  \label{le:54}
  Assume that for the primary method, $Ac^{k-1} = \frac1k c^k$, where
  the power is intended componentwise on the vector elements.
  With the same notation as above, if $s_1\geq s_2\geq 1$, we have
  \begin{eqnarray}
    \label{eq:39}
    \tilde A c^{k-1} &=& \frac{\tilde c^k}{k}, \qquad \hbox{for } k=1, \ldots s_1-1.
  \end{eqnarray}
  In particular, $\sum_{j=1}^{s_1} \tilde a_{i,j} = \tilde c_i$, $i=1,
  \ldots, s_2$.
  Moreover, if:
  \begin{enumerate}
  \item  the  quadrature formula based on the nodes $\tilde b_i$ is exact for
  polynomials of degree at least $s_1-1$ for the interpolation
  \R{eq:18} and the primary method satisfies $\sum_i b_i a_{i,j} = b_j
  (1-c_j)$ for all $j$; and
  \item the quadrature formula based on $\tilde b_i$ and $b_i$
  are of order at least $s_1+1$ for the collocation \R{eq:19},
  \end{enumerate}
  then we have 
   \begin{eqnarray}
    \widehat{\tilde A} \one_{s_2} &=& c, 
                                      \label{eq:41}
   \end{eqnarray}
   that is $\sum_{j=1}^{s_2} \widehat{\tilde a}_{i,j} = c_i$,
   $i=1\ldots, s_1$.
  
\end{lemma}
\begin{proof}
  We first prove \R{eq:39} in the case $k=1$ ($c^0 = \one_{s_1}$).
 
  For the interpolative scheme \R{eq:17},  we have
  \begin{displaymath}
    \sum_{j=1}^{s_1} \tilde a_{i,j} = \sum_{j=1}^{s_1}
                                        \sum_{l=1}^{s_1} \LL_l(\tilde
                                        c_i) a_{l,j} =
                                        \sum_{l=1}^{s_1} \LL_l(\tilde
                                        c_i) \sum_j^{s_1} a_{l,j} = \sum_{l=1}^{s_1} \LL_l(\tilde
                                        c_i) c_l = \tilde c_i
    \end{displaymath}
    where the second last passage holds provided that $\sum_{j}a_{l,j}
    = c_l$, which is true as long as 
    the primary method has order at least one.                                  
   The function $\sum_{l=1}^{s_1} \LL_l(t) c_l$
   is the interpolant at $c_1, \ldots, c_{s_1}$ of the function with
    values $c_1, \ldots, c_{s_1}$ and therefore it is 
    the identity function: $ \sum_{l=1}^{s_1} \LL_l(t) c_i=t$. Evaluating this function
    in $\tilde c_i$ completes the proof of the statement.

  The proof for $k>1$ for the interpolative methods
  follows by a similar argument  as for $k=1$, using the property
  of the primary RK method that $A c^{k-1} = \frac{c^k}{k}$ and the
  fact that the $\LL_i$ are interpolating polynomials on $s_1$ nodes
  interpolating exactly up to degree $s_1-1$.

  For the collocative
  stages \R{eq:19}, we have
  \begin{displaymath}
    (\tilde A c^{k-1})_i = \sum_j \int_0^{\tilde c_i} \LL_j(\tau) c_j^{k-1} =
    \int_0^{\tilde c_i}  \sum_j \LL_j(\tau) c_j^{k-1} = \int_0^{\tilde
      c_i} \tau^{k-1} \D \tau = \frac1k \tilde c_i^k.
  \end{displaymath}
  since the $\LL_j$ interpolate exactly polynomials up do degree
  $s_1-1$  as above.

  For the proof of \R{eq:41}, we observe that
  \begin{eqnarray}
    \widehat{\tilde A} \one_{s_2} &=& (\one_{s_1\times s_2} - B^{-1}
                                      \tilde A^T )\tilde  B \one_{s_2} =
                                      (\one_{s_1\times s_2} - B^{-1}
                                      \tilde A^T ) \tilde b \nonumber\\
                                  &=&\one_{s_1} - B^{-1} \tilde A^T
                                      \tilde b,  \label{eq:42}
  \end{eqnarray}
  where in the last passage we have used the fact that $\sum \tilde
  b_i=1$.
  In the interpolative setting \R{eq:18}, we look at the term $B^{-1} \tilde A^T\tilde b= B^{-1} A^T \LL(\tilde
  c)^T \tilde b$. By construction,
\begin{displaymath}
  (\LL(\tilde c)^T \tilde b)_i = \tilde b_1 \LL_i(\tilde c_1) + \cdots +
  \tilde b_{s_2} \LL_i(\tilde c_{s_2}) = \int_0^1 \LL_i(\tau) \D \tau =
  b_i
\end{displaymath}
provided that the quadrature formula based on the nodes $\tilde b_i$
is exact for polynomials of degree at least $s_1-1$. It follows that
$\LL(\tilde c)^T \tilde b = b$. Further, we have $\tilde A^T b = B
(\one_{s_1} - c)$ because of the property of the primary RK
method. Thus, $B^{-1} \tilde A^T\tilde b= B^{-1} A^T \LL(\tilde
  c)^T \tilde b = \one_{s_1}-c$, which, substituted in \R{eq:42}
  completes the proof.

  In the collocative setting \R{eq:19},
  \begin{eqnarray*}
    (\tilde A^T \tilde b)_i &=& \sum_{j=1}^{s_2} \tilde a_{j,i} \tilde
                                b_j = \sum_{j=1}^{s_2} \tilde b_j \int_0^{\tilde c_j} \LL_i(\tau) \D \tau\\
                            &=& \sum_{j=1}^{s_2}\tilde b_j  f_i (\tilde
                                c_j), \qquad f_i (t) = \int_0^t
                                \LL_i(\tau) \D \tau\\
                                &=& \int_0^1 f_i(t) \D t
  \end{eqnarray*}
  since the $f_i$s are polynomials of degree $s_1$ and the quadrature
  formula based on the nodes $\tilde b_i$ has order at least $s_1$.
  Then $(B^{-1}  \tilde A^T \tilde b)_i = \frac1{b_i}  \int_0^1
  \int_0^{t} \LL_i(\tau)\D \tau \D t$. Applying integration by parts,
  $\int_0^1  \int_0^{t} \LL_i(\tau)\D \tau \D t = [t \int_0^t \LL_i(\tau)\D
  \tau]_{0}^1 - \int_0^1 t \LL_i(t) \D t = b_i - \sum_{j} b_j
  c_j\LL_i(c_j) = b_i - b_i c_i$. The second last passage follows
  provided that the integration formula with weights and nodes $(b,c)$ is exact for
  polynomials of degree $s_1+1$ and from $\LL_i(c_j) = \delta_{i,j}$.  Thus $B^{-1}
  \tilde A^T \tilde b  = \one_{s_1} -c$, which, substituted in
  \R{eq:41}, completes the proof.
\end{proof}

\begin{theorem}
  \label{th:order}
  Consider the methods \R{eq:20} under the conditions of
  Lemma~\ref{le:54}. Assume that $(b,c)$ and $(\tilde b, \tilde c)$
  $s_1$ and $s_2\leq s_1$ quadrature nodes and 
  weights of a quadrature formula of order at least $r\geq s_1$, so that
  \begin{equation}
    b^T c^n = \tilde b^T \tilde c^n = \frac1{n+1}, \qquad n= 0,
    \ldots, r
    \label{eq:quad_ord}
  \end{equation}
  (the power is intended componentwise). Then the
  interpolative \R{eq:18} and collocative \R{eq:19} methods \R{eq:20}
  have also order $r$.
\end{theorem}
\begin{proof}
To prove the theorem it is sufficient to show that that quadrature formula 
\emph{interpolating} the nodes $\tilde c$ using 
  the nodes $c$,
  \begin{equation}
    \int_0^1 f(x) \D x \approx \sum_{i=1}^{s_2} \tilde b_i
   \tilde f_i \qquad \tilde f_i = \sum_{j=1}^{s_1} \LL_j (\tilde c_i) f(c_j)
    \label{eq:quad_interp}
  \end{equation}
  as well as the quadrature formula \emph{collocating} the
  nodes $\tilde c$,
  \begin{equation}
    \label{eq:quad_coll}
    \int_0^1 f(x) \D x \approx \sum_{i=1}^{s_2} \tilde b_i
   \tilde f_i \qquad \tilde f_i = \int_0^{\tilde c_i} \sum_{j=1}^{s_1}
   \LL_j (x) f'(c_i) \D x
  \end{equation}
  have also order $r$ when $f(x) = x^n$,
  $n=0,\ldots,r$, for which $\int_0^1 x^{n} \D x = \frac1{n+1}$.

 We start with proving \R{eq:quad_interp}  for the interpolative formulas.
  For $n=0, \ldots, s_1-1$ the statement is immediate as
the function $\sum_j \LL_j(x) f(c_j)$ exactly
interpolates polynomials of degree up to degree $n=s_1-1$, hence
$\sum_j \LL_j(\tilde c_i) c_j^n = \tilde c_i^n$. Hence by virtue of
\R{eq:quad_ord} the statement follows.

When $n = s_1, \ldots, r$, note that $\sum_{i=1}^{s_2} \tilde b_i
   \tilde x_i^n = \tilde b^T \tilde x^n = \tilde b^T \LL(\tilde c)
   c^n$, where $\tilde x^n = \LL(\tilde c) c^n$. As shown 
   in Lemma~\ref{le:54}, $\tilde b^T \LL(\tilde c) = b^T$, hence  $\sum_{i=1}^{s_2} \tilde b_i
   \tilde x_i^n =  b^T c^n=\frac1{n+1}$ and 
   the statement follows from the assumption \R{eq:quad_ord}.

   For the collocative formulas and \R{eq:quad_coll}, when $f(x) = x^n$, we have $f'(x) =
   nx^{n-1}$ so that the interpolation $\sum_j \LL_j(x) c_j^{n-1} =
   x^{n-1}$ is exact for polynomials of degree $n=0, \ldots,
   s_1$. Consequently, $\tilde f_i = \tilde c_i^n$ and the statement
   follows. When $n = s_1+1, \ldots, r$, we refer again to the
   computations in Lemma~\ref{le:54}: $(\tilde b^T \tilde A)_i =
   \int_0^1 \int_0^T \LL_i (\tau) \D \tau= b^T (I-\diag c)$ (the last
   passage follows integrating by part). Therefore
   \begin{eqnarray*}
     \tilde b^T \tilde f &=& b^T (I-\diag c) n c^{n-1} \\
     &=& n b^T c^{n-1} - n b^T c^n = n \frac1n - n \frac{1}{n+1} = \frac1{n+1},
   \end{eqnarray*}
   which completes the proof.
 \end{proof}
 We are especially interested on the family of methods generated by
 the Lobatto IIIA-B (primary method) and Gauss-Legendre (secondary
 method) of the same order ($r=2s_2=2(s_1-1)$). These quadrature formulas
 are superconvergent and the proof of superconvergence  is heavily
 based on the roots and 
weights of the corresponding orthogonal polynomials, so that, in
principle, the interpolation might destroy the super
convergence. Fortunately, this does not happen because the methods
satisfy the hypotheses of Theorem~\ref{th:order}, and the order is
preserved. This statement is summarized in the Corollary below.
 \begin{coroll}
  The methods \R{eq:20} with primary method Lobatto IIIA-B with $s_1$
   stages and secondary method Gauss--Legendre
   with $s_2 = s_1-1$ stages has order
   $r=2(s_1-1)=2s_2$ both for coefficients based on interpolation  and collocation.
\end{coroll}

\section{P-stability}
\label{sec:P-stability}

P-stability is a desirable property when applying a numerical method to highly
oscillatory systems. The test model is the harmonic oscillator
\begin{equation}
  \label{eq:25}
  \begin{bmatrix}
    q'\\p'
  \end{bmatrix} =
  \begin{bmatrix}
    0&1 &  \\ -\omega^2 & 0 
  \end{bmatrix} \begin{bmatrix}
    q\\p
  \end{bmatrix}, \qquad \omega \in \RR^+,
\end{equation}
whose exact solution can be written as
\begin{equation}
  \label{eq:26}
  \begin{bmatrix}
    q(t_0+h)\\
    p(t_0+h)
  \end{bmatrix} = D_\omega \Theta(\mu) D_\omega^{-1}
  \begin{bmatrix}
    q(t_0)\\p(t_0)
  \end{bmatrix},\quad \Theta(\mu) =
  \begin{bmatrix}
    \cos \mu & \sin \mu \\ -\sin \mu &\cos\mu,
  \end{bmatrix}, \qquad  \mu = \omega h
\end{equation}
where $D_\omega = \diag{1, \omega}$.
It is well known that the application of a $s$-stages PRK pair with coefficient $(A,
b)$ and $(\widehat A, b)$ yields a numerical approximation 
\begin{equation}
  \label{eq:27}
  \begin{bmatrix}
    q_1\\
    p_1
  \end{bmatrix} = D_\omega M(\mu) D_\omega^{-1}
  \begin{bmatrix}
    q_0\\p_0
  \end{bmatrix}
\end{equation}
with $2\times 2$ stability matrix $M(\mu)$
\begin{equation}
  \label{eq:28}
  M(\mu) = I_2 + \mu
  \begin{bmatrix}
    O & b^T\\ -b^T & 0 
  \end{bmatrix}
  \begin{bmatrix}
    I_s & -\mu A\\
    \mu \widehat A & I_s
  \end{bmatrix}^{-1}
  \begin{bmatrix}
    \one_s & O\\ O & \one_s
  \end{bmatrix}.
\end{equation}
We are interested in methods that preserve the unit modulus of the
eigenvalues of the rotation matrix $\Theta(\mu)$.
\begin{definition}
  A numerical method is \emph{P-stable} if for all $\mu \in \RR$ the
  eigenvalues $\lambda_i(\mu)$, $i=1,2$ of $M(\mu)$ satisfy
  \begin{itemize}
  \item $|\lambda_i(\mu)| = 1$, $i=1,2$ and $\lambda_1(\mu)\not=
    \lambda_2(\mu)$; or
  \item $\lambda_1(\mu) = \lambda_2(\mu) = \pm1$ and the eigenvalues
    possesses two distinct eigenvectors.
  \end{itemize}
\end{definition}
It is well known  that symmetric RK methods are P-stable, and, as a
consequence, the methods Lobatto IIIA and Lobatto IIIB, taken individually, are
P-stable. However, the PRK combination Lobatto IIIA-B, which
include the Verlet scheme for order 2, \emph{is not  P-stable}
\cite{jay95hos,mclachlan11lso}.
Motivated by the positive results of the IMEX method, that was proven
to be P-stable (\emph{unconditionally stable}, \cite{mclachlan14mti}),
we study the methods \R{eq:20} and, the same spirit of the
IMEX methods, the oscillatory part is treated by the secondary method
(i.e.\ we set $F^1=0$).

\begin{theorem}{}
  The matrix $M(\mu)$ for the method \R{eq:20} is given as
  \begin{equation}
    \label{eq:Mlambda}
    M(\mu) = I_2 + \mu
    \begin{bmatrix}
      0 & b^T \\ - \tilde b^T & 0
     \end{bmatrix} 
      \begin{bmatrix}
        I_{s_2} & -\mu \tilde A \\
        \mu \widehat{\tilde A} & I_{s_1} 
      \end{bmatrix}^{-1}
      \begin{bmatrix}
        \one_{s_2} & 0\\
        0 & \one_{s_1}
      \end{bmatrix}, \qquad \mu = \omega h.
    \end{equation}
    Moreover, as the methods are symplectic, 
    \begin{equation}
      \label{eq:detMlambda}
      \det M(\mu)=1.
    \end{equation}
\end{theorem}
\begin{proof}
  For the test equation \R{eq:25} ($F^{1}=0$), the method \R{eq:20} can be written
  as
  \begin{eqnarray*}
    P_i &=& p_0 - h \omega^2 \sum_{j=1}^{s_2} \widehat{\tilde a}_{i,j}
            \tilde Q_j\\
    \tilde Q_i &=& q_0 + h \sum_{j=1}^{s_2} \tilde a_{i,j} P_j\\
    p_1 &=& p_0 - h \omega^2 \sum_{j=1}^{s_2}\tilde b_j  \tilde Q_j \\
    q_1&=& q_0 + h \sum_{j=1}^{s_1} b_j P_j.
  \end{eqnarray*}
To ease notation, we denote by capital letters $\tilde Q$ the vector of the internal
stages $\tilde Q_i$, by $P$ the vector of the internal momenta $P_i$,
and abuse notation, to avoid the use of tensor products. So, for
instance, $\tilde A P$ has, as an $i$-component, the vector $\tilde
a_{i,1} P_1 + \cdots +\tilde a_{i,s_1} P_{s_1}$.

In block form, we have
\begin{displaymath}
  \begin{bmatrix}
    I & -h\tilde A\\
    \omega^2 h \widehat{\tilde A}& I
  \end{bmatrix}
  \begin{bmatrix}
    \tilde Q \\ P 
  \end{bmatrix} =
  \begin{bmatrix}
    q_0 \\
    p_0 
  \end{bmatrix}
\end{displaymath}
which we use to solve for the $\tilde Q$ and $P$.
From $q_1 = q_0 + h b^T P $ and $p_1 = p_0-h \omega^2
\tilde b^T \tilde Q$, we get
\begin{eqnarray}
  \begin{bmatrix}
    q_1\\ p_1 
  \end{bmatrix} &=&
  \begin{bmatrix}
    q_0 \\ p_0 
  \end{bmatrix} + h
  \begin{bmatrix}
    0 & b^T \\ -\omega^2 \tilde b^T & 0 
  \end{bmatrix}
  \begin{bmatrix}
    I & - h\tilde A \\ h\omega^2 \widehat{\tilde A}& I
  \end{bmatrix}^{-1}
  \begin{bmatrix}
    q_0 \\ p_0 
  \end{bmatrix}
  \nonumber \\
  &=& D_\omega M(\mu) D_\omega^{-1}\!\!
      \begin{bmatrix}
        q_0\\ p_0
      \end{bmatrix}
\end{eqnarray}
where the last passage follows in a manner very
  similar as corresponding proof for PRK methods with $M(\mu)$ as in \R{eq:Mlambda}. 

  As for \R{eq:detMlambda}, if the method is
  symplectic, then it must be volume preserving for Hamiltonian
  systems, which, in this case implies that  $\det M(\mu)=1$.
\end{proof}
Since the eigenvalues of the matrix $M(\mu)$ are
\begin{displaymath}
  \lambda_i(\mu) = \frac12 \tr M \pm \sqrt {(\frac12 \tr M)^2
    - \det M}, \qquad i=1,2,
\end{displaymath}
because of the determinant condition \R{eq:detMlambda} one has
$\lambda_1 \lambda_2=1$. Hence the eigenvalues lie on the unit circle
if and only if 
\begin{equation}
  \label{eq:30}
  |\tr M(\mu)| \leq 2.
\end{equation}
In addition, when $\tr M =2$, the eigenvalues are both equal to $1$,
while for $\tr M = -2$, the eigenvalues are both equal to $-1$. When
studying P-stability, we will refer to the function
\begin{displaymath}
  \frac 12 | \tr M(\mu)|
\end{displaymath}
as \emph{stability function} of the method.

Provided that the method is P-stable, it can be interpreted as an
oscillator with a modified frequency. Comparing with the matrix
$\Theta(\mu)$ in \R{eq:26}, we have
\begin{equation}
  \label{eq:34}
  \frac 12 \tr M(\mu) = \cos (\tilde \omega h)=\cos(\tilde \mu),
  \qquad \tilde \mu = \tilde \omega h
\end{equation}
corresponding to  a modified frequency $\tilde \omega$ satisfying
\begin{equation}
  \label{eq:43}
  \tilde \mu= \tilde \omega h=\arccos(\frac12 \tr M (\mu)), \qquad \mu=\omega h.
\end{equation}

\begin{coroll}
  The IMEX method is P-stable.
\end{coroll}
\begin{proof}
  By direct computation, the IMEX method has stability matrix
  \begin{displaymath}
    M(\mu) = \frac{1}{1+\nu^2}
    \begin{bmatrix}
      1-\nu^2 & \mu\\
      -\mu & 1-\nu^2
    \end{bmatrix}, \qquad \nu = \frac{\mu}2,
  \end{displaymath}
  with trace $\tr M = 2 \frac{1-\nu^2}{1+\nu^2}$ which always
  satisfies \R{eq:30}. 
\end{proof}
The modified frequency of the IMEX is thus $\tilde \omega = \frac1h \arccos(
\frac{1-\lambda^2 h^2/4}{1+\lambda^2 h^2/4})$, as already found in \cite{mclachlan14mti}.

Because of the symplecticity of the methods \R{eq:20}, it is obvious that in
order to study the P-stability it is sufficient to
look at the diagonal elements $M_{1,1}$ and $M_{2,2}$ of the matrix
$M(\mu)$ in \R{eq:Mlambda}. By direct computation, one has that
\begin{eqnarray}
  M_{1,1} &=& 1-\mu^2 b^T \widehat{\tilde A} (I_{s_2}+\mu^2\tilde A
              \widehat{\tilde A})^{-1} \one_{s_2}
              \label{eq:M11}\\
  M_{2,2} &=&  1-\mu^2 \tilde b^T \tilde A (I_{s_1}+\mu^2
              \widehat{\tilde A} \tilde A)^{-1} \one_{s_1}.
              \label{eq:M22}
\end{eqnarray}

\begin{lemma}
  \label{th:M11eqM22}
  Under the requirements of the Lemma~\ref{le:54}, \R{eq:M11}-\R{eq:M22} can be
written as  
\begin{eqnarray}
  \label{eq:M11_c}
  M_{1,1} &=& 1-\mu^2 b^T (I_{s_1} + \mu^2 \widehat{\tilde A} \tilde
              A)^{-1} c \\
  \label{eq:M22_ct}
  M_{2,2} &=& 1- \mu^2 \tilde b^T (I_{s_2} + \mu^2 \tilde
              A\widehat{\tilde A})^{-1} \tilde c.
\end{eqnarray}
 Moreover, if
  \begin{equation}
  \label{eq:46}
  b^T (\widehat{\tilde A} \tilde A)^k c = \tilde b^T (\tilde A
  \widehat{\tilde A})^k \tilde c, \qquad k =0, \ldots, \min\{s_1,s_2\}-1,
\end{equation}
then $M_{1,1} = M_{2,2}$.
\end{lemma}
\begin{proof}
 We use the formal series $(I+G)^{-1} = \sum_k (-1)^k G^k$.
 The first part of the statement says that we can push
 $\widehat{\tilde A}$ and $\tilde{A} $ on the other right hand side using \R{eq:39} and
\R{eq:41} from Lemma~\ref{le:54}, that is $\tilde{A} \one_{s_1} = \tilde c$  and
$\widehat{\tilde A} \one_{s_2}= c$. 
  
 For the second part of the statement, if all the infinite
 terms of the series  in $M_{1,1}$ and $M_{2,2}$  are equal for
 $k=1,2, \ldots$, then the series are also equal, even
 if the series do not converge. To prove this, note that the matrices $(\tilde A
  \widehat{\tilde A}) $ and $(\widehat{\tilde A}\tilde A) $ have the
  same $n$ nonzero eigenvalues $\lambda_1, \ldots,
  \lambda_{n}$, $n\leq \min\{s_1, s_2\}$. By the Cayley--Hamilton
  theorem,  $G^{m}$ can obtained as a linear combination of $I,
  \ldots, G^{n-1}$ for  $m\geq n$. Therefore only the terms in
  \R{eq:46} need be checked.
\end{proof}
\textbf{Remark.} Note that for $k=0$, we have $b^Tc = \frac12 = \tilde b^T \tilde c$ is
always verified for methods of order at least one.

The combination Lobatto IIIA-B and Gauss-Legendre of the same
order satisfies the requirements of Lemma~\ref{th:M11eqM22}, hence it
is sufficient to check \R{eq:46} only up to $k=1$ (method of order 4)
and $k=2$ for the method of order six.

We show the verifications for the methods based on
\emph{interpolation}. 
For order 4, the proof is immediate for all $k$ because $\tilde c =
\frac12 \one_2$), hence
\begin{eqnarray}
  b^T (\widehat{\tilde A} \tilde A)^k c &=&  b^T \widehat{\tilde A}
                                            (\tilde A \widehat{\tilde
                                            A} )^{k-1} \tilde A c
                                            \nonumber\\
   &=& \tilde b^T (I - \diag {\tilde c}) (\tilde A \widehat{\tilde
                                            A} )^{k-1} \tilde
       A\widehat{\tilde A} \one_2  \label{eq:47} \\
  &=& \frac12 \tilde b^T  (\tilde A \widehat{\tilde
                                            A} )^{k} \one_2  \nonumber
  \\
  &=& \tilde b^T (\tilde A \widehat{\tilde
                                            A} )^{k} \tilde c \nonumber
\end{eqnarray}
where we have used $ b^T \widehat{\tilde A} = \tilde b^T (I - \diag
{\tilde c})$ and \R{eq:41}.
When going to higher order, a general proof of $M_{1,1} = M_{2,2}$
using an argument as above doesn't seem
straightforward because of in general $\widehat{\tilde A}\tilde
c^{k-1} \not= \frac{c^k}{k}$ for $k>1$ (see \R{eq:41}).
Yet, \R{eq:46} can be verified by direct computation. For
instance, for the order six combination
\begin{align}
  &\tilde b^T \tilde A \widehat{\tilde A} \tilde c =   b^T
                                                      \widehat{\tilde
    A} \tilde A  c =\frac1{24},\nonumber \\
  &\tilde b^T (\tilde A \widehat{\tilde A})^2 \tilde c =   b^T
                                                      (\widehat{\tilde A} \tilde A)^2  c =\frac1{720}.\nonumber
\end{align}

Our preliminary numerical tests seem to confirm that Lemma~\ref{th:M11eqM22}
yields for a larger class of methods, therefore we conjecture that
$M_{1,1}= M_{2,2}$ whenever the secondary quadrature has order
at least equal to the order of the primary method. 

\begin{theorem}
  The methods \R{eq:20} based on  Lobatto IIIA and Gauss--Legendre of
  order four and six with $\tilde A$ \emph{by interpolation} \R{eq:17} are
  \emph{P-stable} and correspond to oscillators with modified
  frequencies. These are
  \begin{equation}
    \label{eq:35}
    \tilde \omega h = \tilde \mu= \arccos \left(\frac{1-\frac5{12}\mu^2
      + \frac1{144}\mu^4}{1+\frac1{12}\mu^2
      + \frac1{144}\mu^4}\right) \qquad
    \mu = \omega h
\end{equation}
for the method of order four. The $\tilde \mu$ touches the line $-1$
at $\mu = 2\sqrt{3}$.

Moreover,
  \begin{equation}
    \label{eq:50}
    \tilde \omega h = \tilde \mu=  \arccos\left( \frac{
        1-\frac{9}{20} \mu^2+\frac{11}{600}\mu^4-\frac{1}{14400}\mu^6}{
        1+\frac1{20} \mu^2 + \frac1{600}\mu^4 + \frac{1}{14400}\mu^6}
       \right) \qquad
    \mu = \omega h
  \end{equation}
  for the methods of order six. The $\tilde \mu$ touches the line $-1$
  at $\mu = \sqrt{10}$ and $1$ at $\mu=2\sqrt{15}$. 
   
  The methods \R{eq:20} based on  Lobatto IIIA-B and Gauss-Legendre of
  order  four and six with $\tilde A$ \emph{by collocation} \R{eq:18} 
  \emph{are not P-stable}.

  The interval of stability in the positive
  half plane are:  $[0,4]$  for the method of order two, $[0,
  \frac6{11}\sqrt{33}]\cup [2\sqrt3, 3 \sqrt6]$ for  
  for the method of order four, and $[0, \sqrt{70-2\sqrt{905}}]\cup
  [\sqrt{10}, \frac85 \sqrt{15}]\cup [2\sqrt{15},  \sqrt{70+2\sqrt{905}}]$ for the methods of
  order six.
\end{theorem}
\begin{proof}
  The methods satisfy Lemma~\ref{th:M11eqM22} therefore one has that
  $M_{1,1}=M_{2,2}$. Taking either 
  of them, the stability functions have been computed using a symbolic
  manipulator, as well as their points of intersections with the lines $\pm1$.
\end{proof}
A plot of the stability functions for the the LobattoIIIA and
Gauss--Legendre combinations by interpolation \R{eq:17} (left)  and
with $\tilde A$ by collocation \R{eq:18} (right) for the methods of order two
(IMEX), order four and order 6 is shown in Fig~\ref{fig:Mlambda}. 
 
\begin{figure}[th]
  \centering
  \includegraphics[width=.45\textwidth]{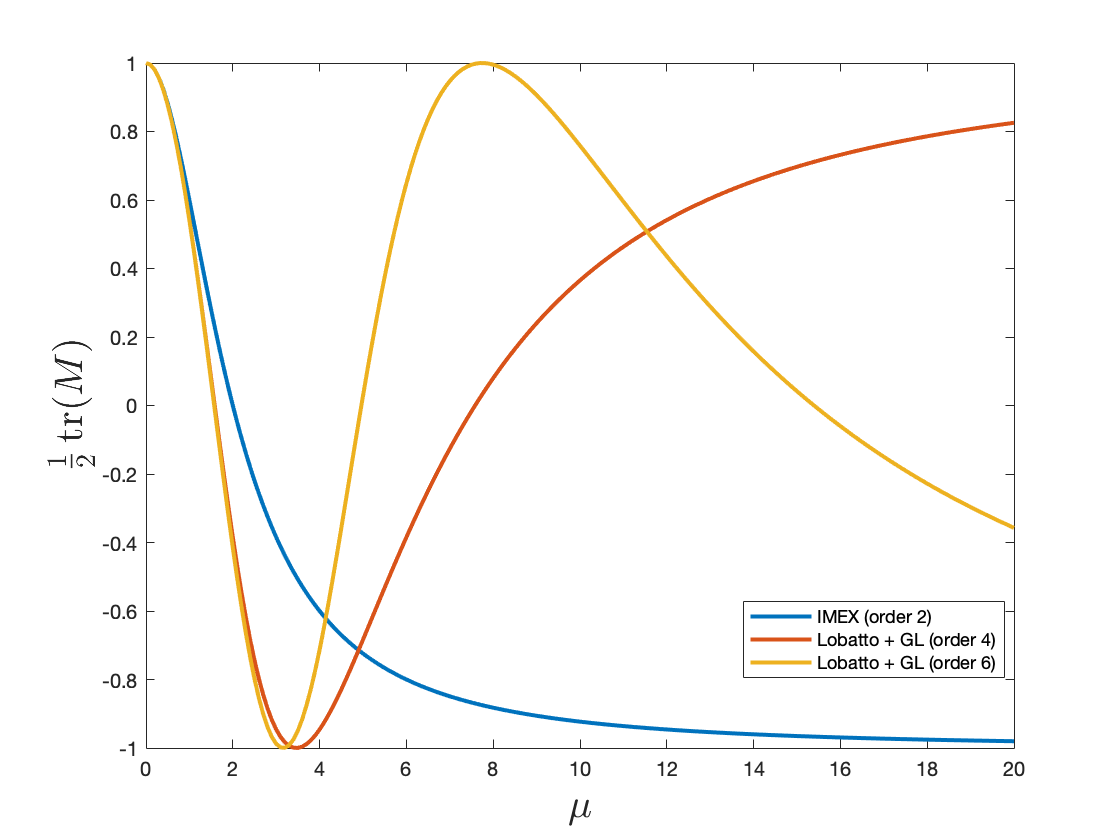}
 \includegraphics[width=.45\textwidth]{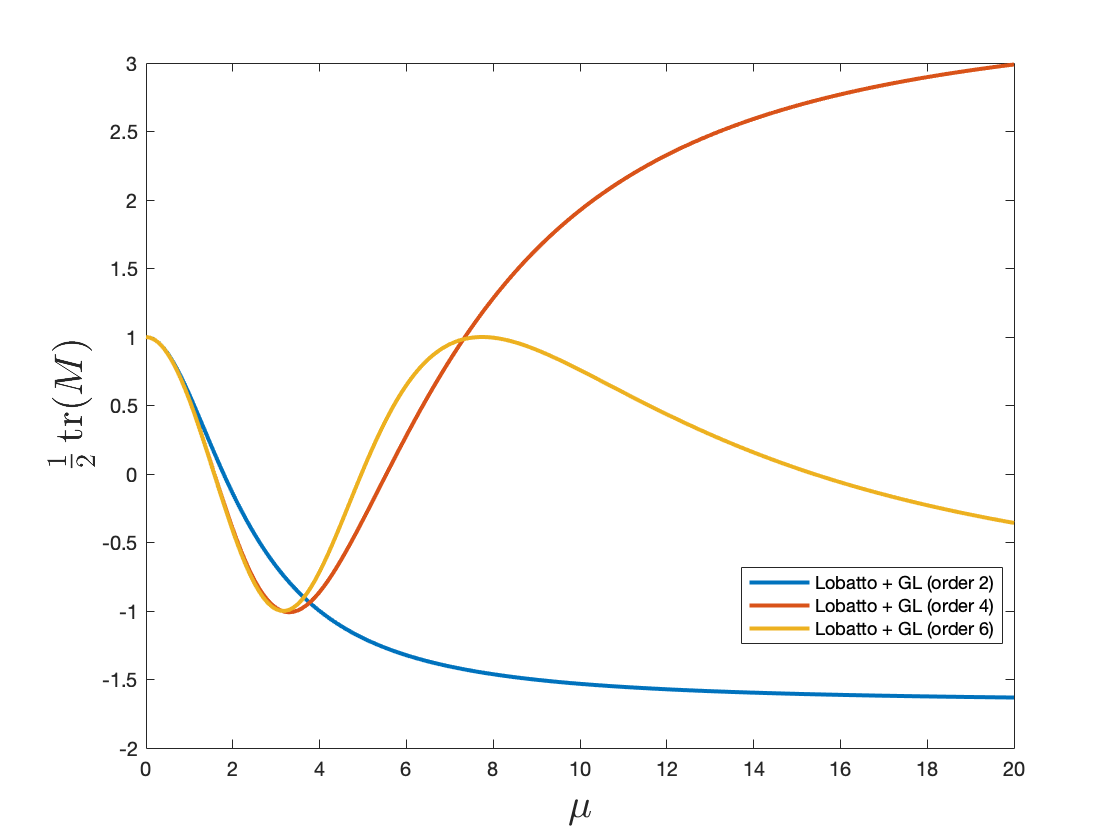}
  \caption{\emph{Left:} Plot of the stability functions for the the Lobatto IIIA-B and
Gauss--Legendre combinations of order two (IMEX), order four and order
six with coefficients constructed by interpolation \R{eq:17}. These
methods are P-stable. \emph{Right:} Stability function plot for the methods with
coefficients constructed by collocation \R{eq:18}. For 
P-stability, the function must have values between $-1$ and $1$ for
all $\mu$. These methods are not P-stable. See text for their interval
of stability.}
 \label{fig:Mlambda}
\end{figure}
\begin{figure}
  \centering
  \includegraphics[width=.5\textwidth]{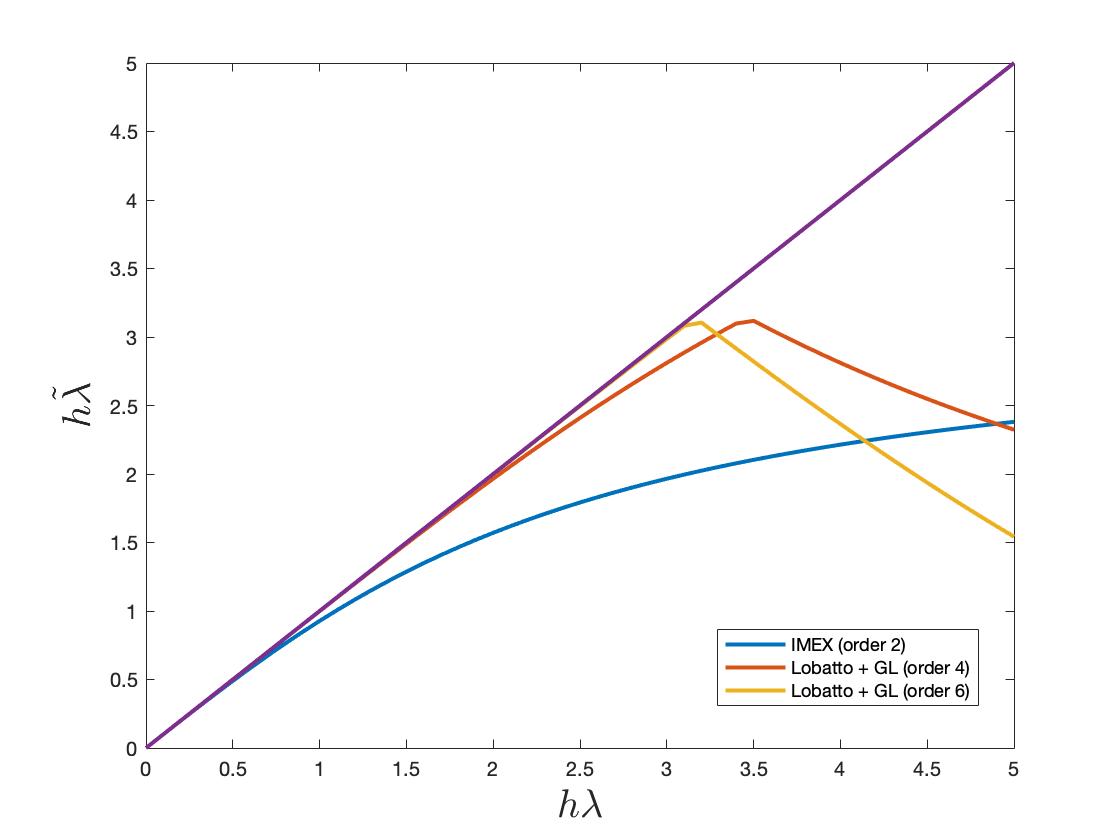}
  \caption{Modified frequency for the methods in the Lobatto IIIA and
    Gauss--Legendre family of order two (IMEX), four and six. The
    straight line is the identity function, for which $\tilde \lambda
    = \lambda$. The IMEX
  retains the correct frequency of oscillations up to $h\lambda\approx 1$. The order
  two method retains the correct frequency up to $h\lambda\approx 2$,
  while the order six method up to $h\lambda \approx 3$.}
  \label{fig:hlambda}
\end{figure}

\section{The methods as modified trigonometric integrators}
\label{sec:methods-as-modified}
We consider the application to the test equation
\begin{equation}
  \label{eq:testeq}
  \ddot q  = - \omega^2 q + f(q), \qquad F^1(q)=f(q),\quad F^2(q)=
  -\omega^2 q.
\end{equation}
\begin{theorem}[Modified trigonometric integrator]
 Consider the symplectic methods \R{eq:17} applied to the test oscillatory problem
 \R{eq:testeq}. Assume that the primary method has symmetric stages and
 that $|\frac12 \tr M(\mu)|\leq 1$, with matrix $M(\mu)$  as in
 \R{eq:Mlambda} having two independent eigenvectors in case of equality. Then the method can be considered
 as a symplectic \emph{modified trigonometric integrator} with modified frequency
 satisfying the implicit relation
 \begin{equation}
   \label{eq:24}
   \cos(\tilde \mu) = \frac12 \tr M(\mu), \qquad \tilde \mu = 
   \tilde \omega h, \mu = \omega h
 \end{equation}
 and can be written in the form
 \begin{equation}
   \label{eq:29}
   q_1- 2\cos (\tilde \mu)q_0 + q_{-1} = h^2\psi_1(\tilde \mu) (f(Q_1)+f(Q_{-1})) +
   \cdots +h^2 \psi_{s_1}(\tilde \mu) (f(Q_{s_1})+f(Q_{-{s_1}})) ,
 \end{equation}
 for $s_1$ implicitly defined filter functions
 \begin{equation}
   \label{eq:36}
   \psi_i (\tilde \mu) = b^T (I_{s_1} +\mu^2 \widehat{\tilde A} \tilde
   A)^{-1} \widehat{A}_{i}, \qquad \mu = \omega h,
 \end{equation}
 where $\widehat{A}_{i}$ is the $i$th column of $\widehat A$. The $p$-variables are reconstructed from the formula
 \begin{equation}
   \label{eq:37}
   2 \frac{\tilde \mu}{\mu} \sinc ({\tilde \mu}) p_0 = q_1-q_{-1}- h^2\psi_1(\tilde \mu) (f(Q_1)-f(Q_{-1}) )+\cdots +h^2 \psi_{s_1}(\tilde \mu) (f(Q_{{s_1}})-f(Q_{-{s_1}})),
 \end{equation}
 where the $\psi_i$ are the same as in \R{eq:36}.
\end{theorem}
\begin{proof}
As in the proof of P-stability, we ease notation and denote by capital letters $\tilde Q$ the vector of the internal
stages $\tilde Q_i$, by $P$ the vector of the internal momenta $P_i$,
by $F(Q)$ the vector of the $f(Q_i)$ and abuse notation, to avoid the
use of tensor products. Thus,  the expression $b^T F(Q)$ means
\begin{displaymath}
  b^T F(Q) = b_1 f(Q_1) + b_2 f(Q_2) +\cdots + b_{s_1} f(Q_{s_1}).
\end{displaymath}
Similarly, for matrix products, the expression $\widehat A F(Q)$ has, as
the $i$-component, the vector $\widehat{a}_{i,1} f(Q_1) + \cdots+
\widehat{a}_{i,s_1} f(Q_{s_1})$, etc.

Proceding as for P-stability, we see that
\begin{displaymath}
  \begin{bmatrix}
    I & -h\tilde A\\
    \omega^2 h \widehat{\tilde A}& I
  \end{bmatrix}
  \begin{bmatrix}
    \tilde Q \\ P 
  \end{bmatrix} =
  \begin{bmatrix}
    q_0 \\
    p_0 + h \widehat A F(Q)
  \end{bmatrix}
\end{displaymath}
which we use to solve for the $\tilde Q$ and $P$.
From $q_1 = q_0 + h b^T P $ and $p_1 = p_0+h b^T F(Q) -h \omega^2
\tilde b^T \tilde Q$, we get
\begin{eqnarray}
  \begin{bmatrix}
    q_1\\ p_1 
  \end{bmatrix} &\!\!\!=\!\!\!&
  \begin{bmatrix}
    q_0 \\ p_0 +h b^T F(Q)
  \end{bmatrix} + h
  \begin{bmatrix}
    0 & b^T \\ -\omega^2 \tilde b^T & 0 
  \end{bmatrix}
  \begin{bmatrix}
    I & - h\tilde A \\ h\omega^2 \widehat{\tilde A}& I
  \end{bmatrix}^{-1}
  \begin{bmatrix}
    q_0 \\ p_0 + h \widehat A F(Q)
  \end{bmatrix}
   \label{eq:33}\\
  &\!\!\!=\!\!\!& D_\omega M(\mu) D_\omega^{-1}\!\!
      \begin{bmatrix}
        q_0\\ p_0
      \end{bmatrix}
  +h \!\!
  \begin{bmatrix}
    0 \\ b^T F(Q) 
  \end{bmatrix} + h^2\!\!
  \begin{bmatrix}
    0 & b^T \\ -\omega^2 \tilde b^T & 0 
  \end{bmatrix}\!\!
  \begin{bmatrix}
    I & - h\tilde A \\ h\omega^2 \widehat{\tilde A}& I
  \end{bmatrix}^{-1} \!\!\!
  \begin{bmatrix}
    0 \\  \widehat A F(Q)
  \end{bmatrix}.
   \nonumber
\end{eqnarray}
Let $ \begin{bmatrix}
    X_1 & X_2 \\ X_3 & X_4
  \end{bmatrix}=\begin{bmatrix}
    I & - h\tilde A \\ h\omega^2 \widehat{\tilde A}& I
  \end{bmatrix}^{-1} 
  $. One has
  \begin{eqnarray*}
    X_1 &=& (I_{s_2}+\mu^2 \tilde A \widehat{\tilde A})^{-1} \\
    X_2&=& h \tilde A (I_{s_1} + \mu^2 \widehat{\tilde A}\tilde A)^{-1}\\
    X_3 &=& -h \mu \widehat{\tilde A} (I_{s_2}+\mu^2 \tilde A
            \widehat{\tilde A})^{-1} \\
    X_4 &=& (I_{s_1} + \mu^2 \widehat{\tilde A}\tilde A)^{-1}.
  \end{eqnarray*}
Thus $q_1$ is given by
\begin{displaymath}
  q_1 = \cos(\tilde \mu) q_0 + h\frac{\tilde \mu}{\mu} \sinc (\tilde
  \mu) p_0 + h^2 b^T (I_{s_1}+ \mu^2\widehat{\tilde A} \tilde A  )^{-1}
  \widehat A F(Q_+), 
\end{displaymath}
where, as above, $\mu = \omega h$ and $\tilde \mu=\tilde \omega h$ is the modified
frequency and $F(Q_+)$ indicates that the internal stages are in
$[0,h]$.
The $\cos(\tilde \mu)$ and $\sinc(\tilde \mu)$ terms come form
$D_\omega M(\mu) D_{\omega^{-1}}$ in the usual way, provided that
$|\frac12 M(\mu)|\leq 1$.
By replacing $h$ with $-h$, we have
\begin{displaymath}
   q_{-1} = \cos(\tilde \mu) q_0 - h\frac{\tilde \mu}{\mu} \sinc (\tilde
  \mu) p_0 + h^2 b^T (I_{s_1}+ \mu^2\widehat{\tilde A} \tilde A  )^{-1}
  \widehat{A}  F(Q_-), 
\end{displaymath}
where, as above, $F(Q_-)$ indicates that indicates that the internal stages are in $[0,-h]$
Taking the sum of $q_{1}$ and $q_{-1}$, we obtain
\begin{displaymath}
  q_1-2\cos(\tilde \mu) q_0 + q_{-1} =  h^2 b^T (I_{s_1}+ \mu^2\widehat{\tilde A} \tilde A  )^{-1}
  \widehat{A} (F(Q_+)+  F(Q_-)),
\end{displaymath}
while subtracting the two expressions, we obtain
\begin{displaymath}
  2 h\frac{\tilde \mu}{\mu} \sinc (\tilde
  \mu) p_0 = q_1-q_{-1} - h^2  b^T (I_{s_1}+ \mu^2\widehat{\tilde A} \tilde A  )^{-1}
  \widehat A (F(Q_+)-F(Q_-)).
\end{displaymath}
 With some simple algebraic manipulations, it is easy to recover the
 filter functions. The theorem statement follows by assuming that the primary
method has symmetric stages.
\end{proof}

\textbf{Remark.} The above theorem is also valid for all the methods
described in the paper in the region where the step size $h$ is such
that $|\frac12 \tr M|\leq 1$.

\vspace{10pt}
When the first node $c_1=0$ then $Q_1=Q_{-1} = q_0$ so the first term on the right hand side of  \R{eq:29} becomes $2\psi_1(\tilde
\mu) f(q_0)$ while it cancels in \R{eq:37}. Moreover, in the case of the
Lobatto primary method, $c_{s_1} =1$ hence $Q_{s_1}= q_1$ and $Q_{-s_1}=
q_{-1}$. However,  the last column of the matrix $\widehat A$ is
zero, and so is the last filter function $\psi_{s_1}$.

For the IMEX method, we have $c_1=0, c_2=1$ ($s_1=2$), hence \R{eq:37} gives
\begin{displaymath}
  2 h \frac{\tilde \mu} {\mu} \sinc (\tilde \mu) p_0 = q_1-q_{-1}.
\end{displaymath}
We have $\psi_2=0$ and
\begin{displaymath}
  q_1 - 2 \cos(\tilde \mu)q_0 + q_1 =  h^2 2 \psi_1(\tilde \mu) f(q_0)
  = h^2\left(1+\frac{\mu^2}4\right)^{-1} f(q_0)
\end{displaymath}
and we recover its expression as a modified trigonometric
integrator 
\begin{displaymath}
  q_1 - 2 \cos(\tilde \mu)q_0 + q_1 =  h^2 \psi(\tilde \mu)
  f(\phi(\tilde \mu)q_0)
\end{displaymath}
with filter functions $\phi=1$, $\psi (\xi) = \cos \xi$ satisfying the
implicit relation $\cos
(\tilde \mu) = (1+\frac{\mu^2}4)^{-1}$, as derived in \cite{mclachlan14mti}.

Similarly, for the order four Lobatto--Gauss-Legendre method,
we have
\begin{eqnarray*}
  2 h \frac{\tilde \mu} {\mu} \sinc (\tilde \mu) p_0 &=& q_1-q_{-1} -
   h^2 \psi_2(\tilde \mu) (f(q_{\frac12}) - (f(q_{-\frac12})) 
\end{eqnarray*}
and
\begin{equation}
  \label{eq:LGL4modint}
  q_1 - 2 \cos(\tilde \mu) q_0 + q_{-1} = h^2 2 \psi_1(\tilde \mu)
  f(q_0) + h^2 \psi_2(\tilde \mu) (f(q_{\frac12}) + f(q_{-\frac12})),
\end{equation}
with filter functions $\psi_{i}$, $i=1,2,3$, satisfying the
implicit relations
\begin{equation}
  \label{eq:LGL4filter}
  \psi_1(\tilde \mu) = \frac{2(-\mu^2+12)}{\mu^4 + 12 \mu^2 + 144},
  \qquad
  \psi_2(\tilde \mu) = \frac{2(\mu^2+24)}{\mu^4 + 12 \mu^2 + 144},
  \qquad
  \psi_3=0
\end{equation}
 ($\phi_i=1$, $i=1,2,3$). The modified frequency is given by \R{eq:35}.

Finally, for the of order 6 Lobatto--Gauss-Legendre method, we have
similar expressions, with filters implicitly defined by
\begin{equation}
  \label{eq:LGL6filter}
  \begin{meqn}
    \psi_1(\tilde \mu) &=&\frac{2\mu^4-140\mu^2+1200}{\mu^6+24\mu^4+720\mu^2+14400}, \\
  \psi_2(\tilde \mu) &=&
  \frac{-(\mu^4+50\mu^2-600)\sqrt5-50\mu^2+3000}{\mu^6
    +24\mu^4+720\mu^2+14400},\\
  \psi_3(\tilde \mu) &=&\frac{(\mu^4+50\mu^2-600)\sqrt5-50\mu^2+3000}{\mu^6
    +24\mu^4+720\mu^2+14400},\\
  \psi_4 (\tilde \mu) &=& 0,
  \end{meqn}
\end{equation}
and modified frequency given by \R{eq:50}.

\section{Numerical experiments}
\label{sec:numer-exper}
As a bed test, we consider the Fermi-Pasta-Ulam-Tsingou (FPUT,
formerly FPU) problem of alternating soft and
stiff springs, that has been extensively used in literature to study
methods for oscillatory problems.
Because of the oscillatory nature of the problem, among all the
methods proposed, we test only those that are 
P-stable, as methods that are not P-stable are likely to produce diverging solution
as soon as the step size leaves the region of P-stability.
Therefore, in what follows, all the numerical experiments are performed with
the Lobatto--Gauss-Legendre family \R{eq:20} with coefficients by
interpolation \R{eq:18}.
We will compare these methods also with higher order integrators
obtained using the IMEX (wich is the Lobatto--Gauss-Legendre method of
order 2) and the Yoshida time stepping technique.

\subsection{The Fermi-Pasta-Ulam-Tsingou problem}
For comparison with \cite{hairer06gni,mclachlan14mti},
we consider the same setup with  $2\ell$ points of unit mass
representing alternating soft nonlinear springs and stiff linear springs. 
Setting $q$ to be the concatenation of slow (index $s$) and fast
(index $f$) position variables,
\begin{displaymath}
  q = [q_{s,1}, \ldots, q_{s,\ell}, q_{f,1}, \ldots, q_{f,\ell}]^T,
\end{displaymath}
and $p$ the corresponding momenta, the Hamiltonian reads
\begin{displaymath}
  \begin{split}
    H(q,p) &= \frac12 \sum_{i=1}^\ell (p_{s,i}^2 + p_{f,i}^2)+
  \frac{\omega^2}2 \sum_{i=1}^\ell q_{f,i}^2 \\
  & \qquad + \frac14 \left[
    (q_{s,1}-q_{f,1})^4 + \sum_{i=1}^{\ell-1} (q_{s,i+1} - q_{f, i+1} -
    q_{s,i}- q_{f,i})^4 + (q_{s, \ell}+ q_{f, \ell})^4 \right].
  \end{split}
\end{displaymath}
In our setup, the nonlinear potential and kinetic energy are treated
with the Lobatto IIIA-B pair, while the linear stiff energy
$\frac{\omega^2}2 \sum_{i=1}^\ell q_{f,i}^2 $ is treated with the
Gauss--Legendre methods based on interpolation.

The total oscillatory energy $I$,
 \begin{displaymath}
   I(q_f,p_f) = \frac12 \sum_{i=1}^\ell p_{f,i}^2 + \frac{\omega^2}2
   \sum_{i=1}^\ell q_{f,i}^2 = I_1+\ldots + I_\ell
 \end{displaymath}
is defined as the sum of the oscillatory energies of each fast spring.
For ease of comparison with the numerical examples in literature, the initial conditions used in the simulations are the same as those
in \cite{hairer06gni,mclachlan14mti}
\begin{displaymath}
  q_{s,1}(0) = 1, \quad p_{s,1}(0) =1, \quad q_{f,1}(0) =
  \omega^{-1}, \quad p_{f,1}(0) =1,
\end{displaymath}
and all the other initial values equal to zero.
In the numerical experiments, we use $\ell=3$.

The left plot in figure \R{fig:OsciHerr} shows the oscillatory energies for each spring
and the total oscillatory energy, comparing the IMEX method (which is
the lowest method in the class) and the higher order proposed method
based interpolation (Lobatto--Gauss-Legendre of order 4 and 6).
The right plot shows the corresponding error in the Hamiltonian energy.
\begin{figure}[t]
  \centering
  \includegraphics[width=.47\textwidth]{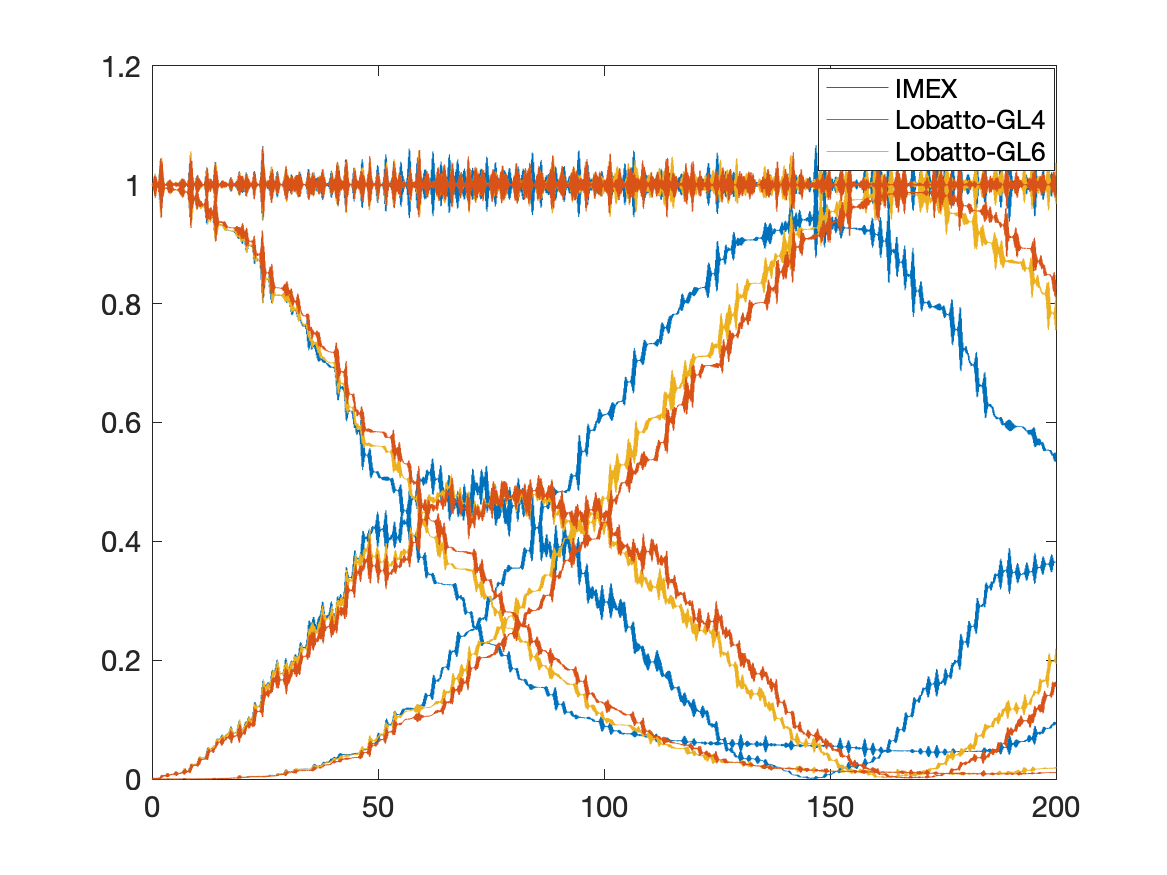}
    \includegraphics[width=.47\textwidth]{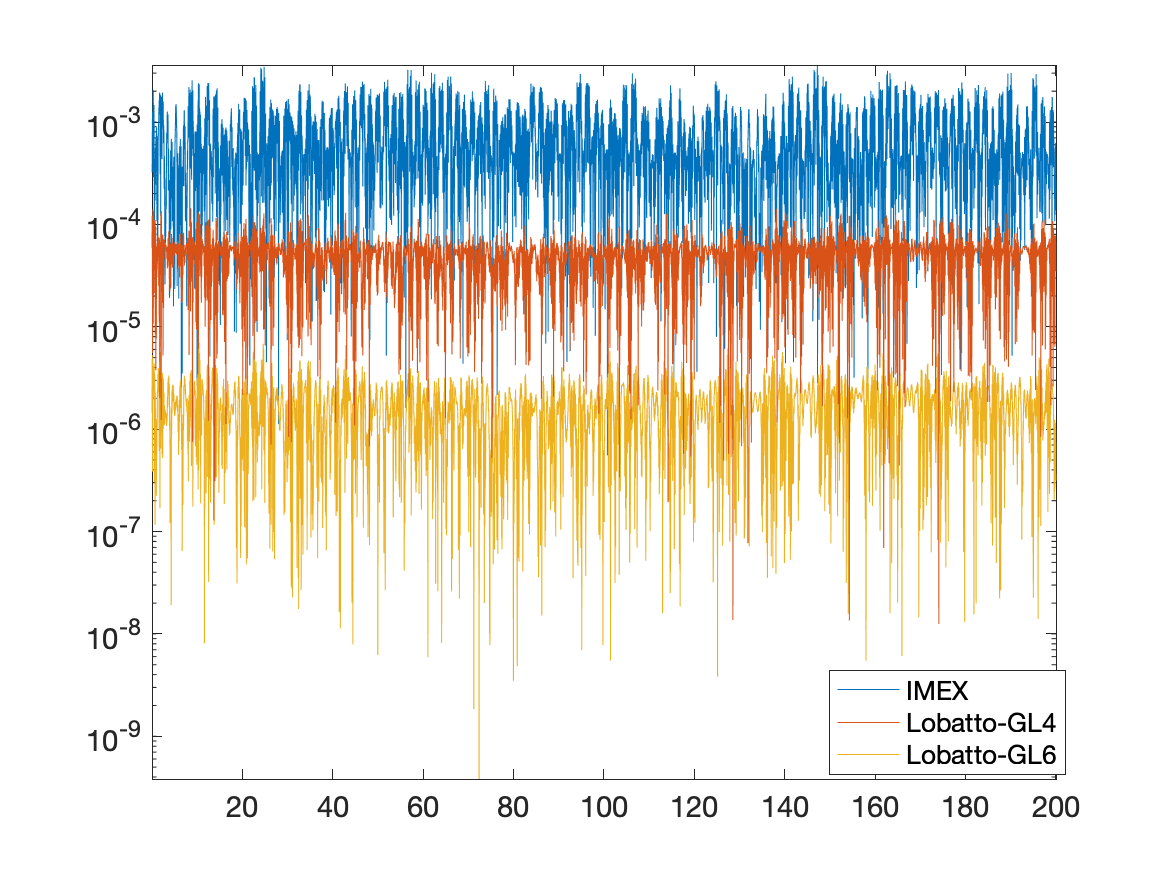}
  \caption{\emph{Left:} Individual oscillatory energies $I_i$ and total oscillatory
    energy $I=\sum_i I_i$. \emph{Right:} Energy error $|H-H_0|$. 
     The simulations are performed in $[0, 200]$ with $\omega = 50$
     and $h=2/\omega=0.04$. See text for initial conditions.}
  \label{fig:OsciHerr}
\end{figure}



When the modified frequency $\tilde \omega$ is such that
$\cos(h\tilde\omega)=\pm1$, see Figure~1, left plot, we expect to
observe resonances.
This happens for $h\omega/\pi =2\sqrt3/\pi\approx
1.1$ for order four method and for $h\omega/\pi = \sqrt{10}/\pi \approx 1$ and
$h\omega/\pi = 2\sqrt{15}/\pi \approx 2.47$ for the order six method.
Resonances can be observed in the preservation of the Hamiltonian
(total) energy of the system and in the scaled total oscillatory
energy $\omega I$ in the range $(0, 4.5\pi]$, the latter being more
uniform in dealing with the frequencies.
It is clear that the width of the resonance region is inversely propotional to
the curvature at the resonance point. The flatter the stability
function is at the resonance points in Figure~ref{fig:Mlambda}, the wider the region of
resonance.

\begin{figure}[t]
  \centering
  \includegraphics[width=.45\textwidth]{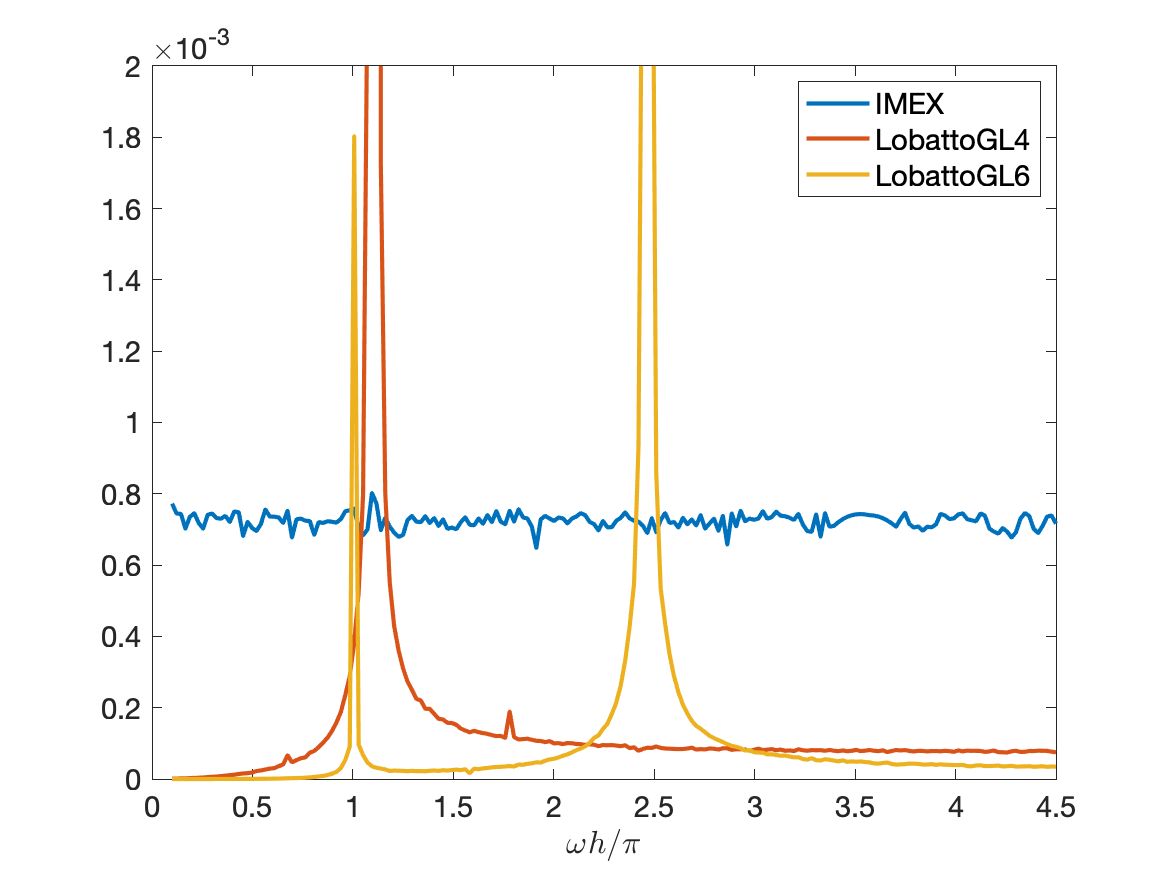}
  \includegraphics[width=.45\textwidth]{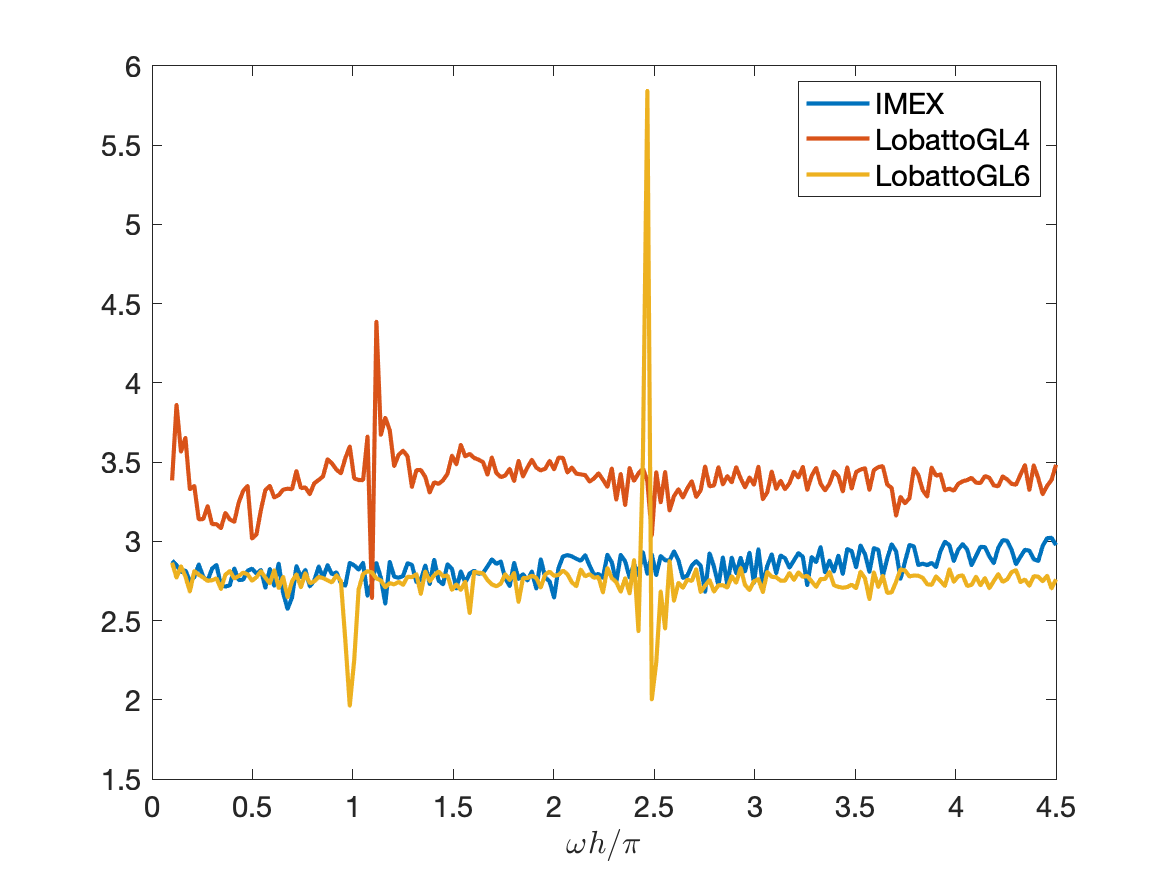}
  \caption{\emph{Left:} Maximum deviation in the Hamiltonian (total)
    energy error. \emph{Right:} Maximum deviation in scaled
    oscillatory energy $\omega I$ error. The computation is performed
    in [0,100] for $h\omega/\pi = 0,...,4.5$,  $h = 0.02$. The peaks
    correspond to the resonances of the methods. These occur when
    $\cos(h\tilde \omega)=\pm1$, namely when $h\omega/\pi \approx
    1.1$ for the order four methods and when $h\omega/\pi \approx 1,
    2.47$ for the order six method. See text for details. }
  \label{fig:scaledI}
\end{figure}

The left plot in Figure~\ref{fig:longtime} displays the solution obtained by the
methods by taking a relatively large step size, with
$h\omega/\pi\approx 1.59$. The approximations to the solutions are
still fairly acceptable and the methods do not display excessive
oscillations as other trigonometric integrators do.
 
The right plot in Figure~\ref{fig:longtime} depicts the behavior of the 
methods as they approach their high-frequency limit, in a similar
experiment as  in \cite{mclachlan14mti}.
We keep $h = 0.1$ but take $\omega = 1000$ with a ratio $h\omega/\pi
\approx 31.8$. In this experiment $\omega$ is scaled by a factor of $20$ (compared to
200 in \cite{mclachlan14mti}) and the time interval must also be
scaled correspondingly to $[0, 4000]$.  
\begin{figure}
  \centering
  \includegraphics[width=.45\textwidth]{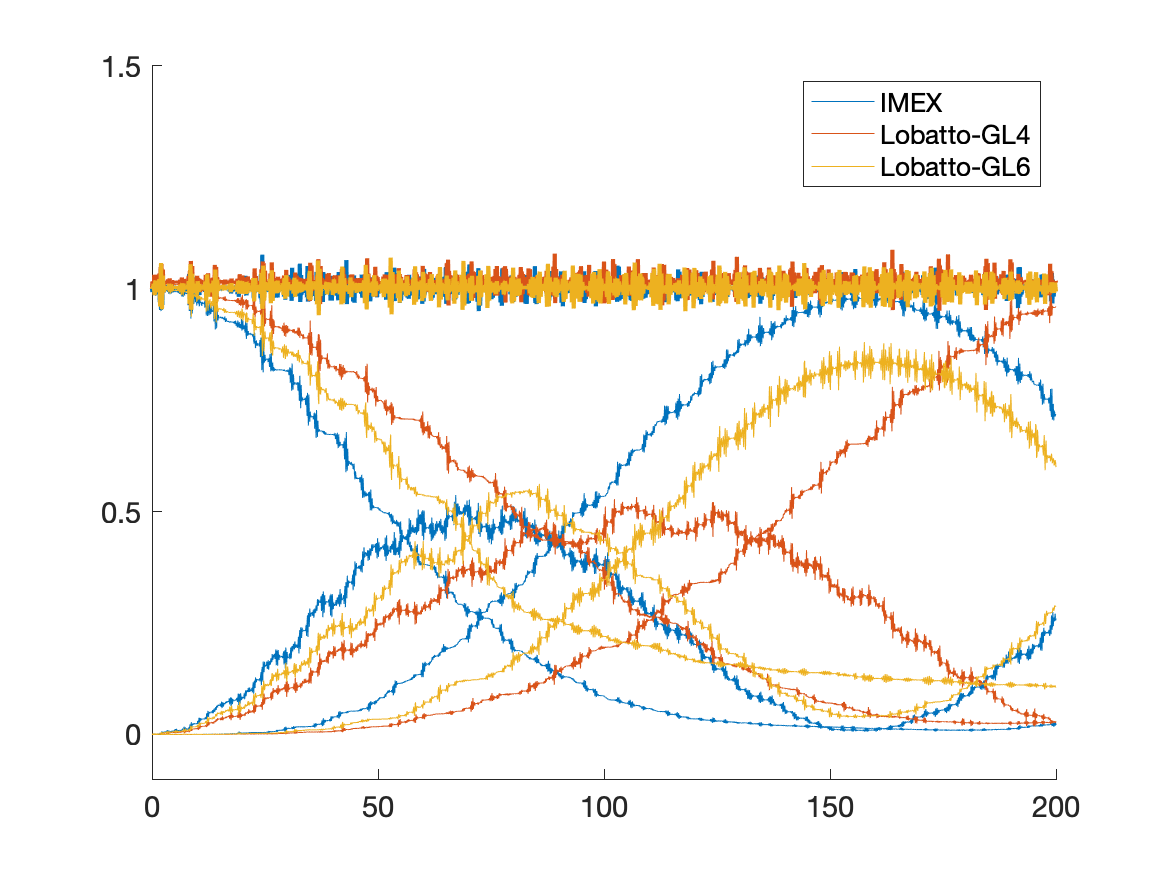}
   \includegraphics[width=.45\textwidth]{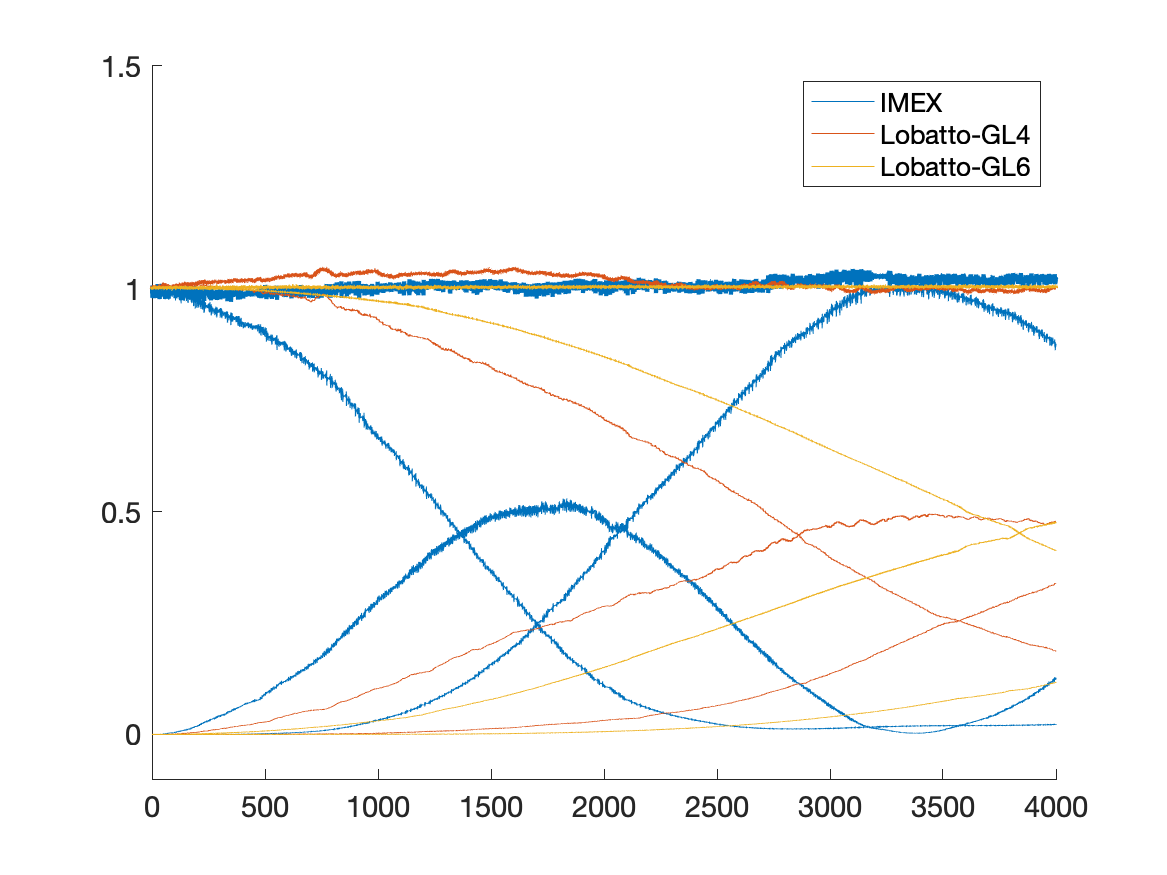}
   \caption{\emph{Left:} Individual oscillatory energies $I_i$ and total oscillatory
     energy $I=\sum_i I_i$, as in Figure~\ref{fig:OsciHerr}, with step
     size  $h=0.1$, $\omega=50$.
     \emph{Right:} Same oscillatory energies as in the left plot. Here the
     step size is kept fixed to $h=0.1$ but the frequency $\omega$ as well as the
     length of the interval is scaled by a factor of $20$.}
   \label{fig:longtime}
 \end{figure}
A rough analysis of the slow energy exchange can be performed  by using the
modified trigonometric integrator form of the method and the expansion
of the exact and numerical solution using modulate Fourier expansions,
see \cite{hairer06gni,mclachlan14mti}. One difficulty
with respect to the standard analysis using modified trigonometric
integrators is the presence of more filter functions $\psi$ in
\R{eq:29} and of the internal stages of the methods. However, performing a Taylor
expansion of the internal stages, one can put the methods in the form
\begin{displaymath}
   q_1- 2\cos (\tilde \mu)q_0 + q_{-1} = h^2\psi(\tilde \mu)
   f(\phi(\tilde \mu) q_0) + \mathcal{O}(h^4)
 \end{displaymath}
 and apply the standard analysis as for trigonometric integrators.
 
 For instance, for the Lobatto--Gauss-Legendre method of order 4, one
 has that $\phi=1$, $\psi(\tilde \mu)=2(\psi_1(\tilde \mu) +\psi_2(\tilde
 \mu))=1/(1+\frac1{12}\mu^2 + \frac1{144} \mu^4)$,
 with $\mu = h\omega$ (the $\psi$-functions are defined
 implicitly).

Using the same setup as in \cite{hairer06gni,mclachlan14mti}, one
finds that $\alpha=1+\frac16\mu^2 +\mathcal{O}(\mu^4)$,
 $\beta=1$, and $\gamma = 1-\frac1{1728}\mu^6 +\mathcal{O}(\mu^8)$. In order to preserve the slow energy exchange
 at a correct rate, it is required that $\alpha=\beta=\gamma=1$, a
 property that is satisfied only by the IMEX, as proven in
 \cite{mclachlan14mti}.
 It is in particular the value of $\alpha$ that has the strongest
 effect on the slow energy exchange. Nevertheless, the methods perform way better
 than classical trigonometric integrators.

\begin{figure}
  \centering
  \includegraphics[width=.6\textwidth]{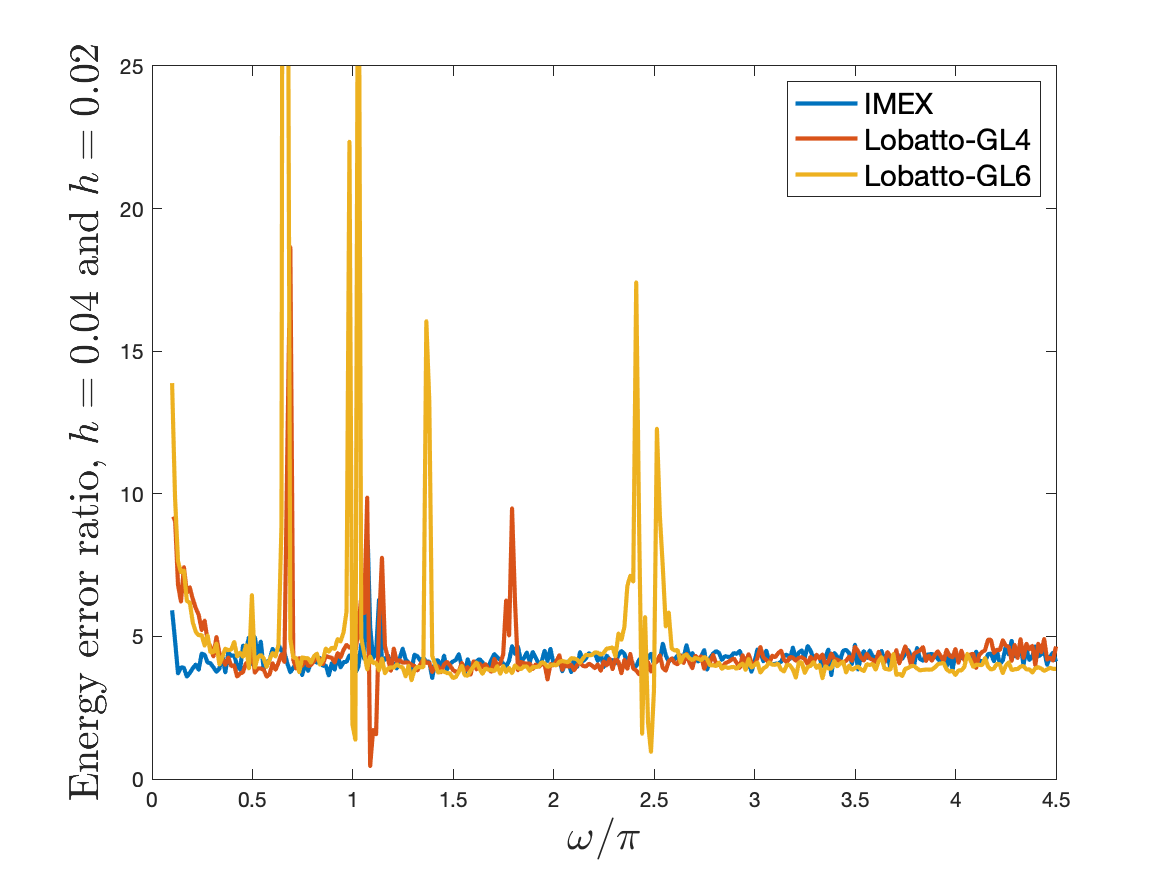}
  \caption{Hamiltonian max error ratio for $h=0.04$ and $h=0.02$,
    $\omega=50$, $T=200$. We observe peaks
    in correspondence of the resonances. }
  \label{fig:Energyerrorratio}
\end{figure}
Figure~\ref{fig:Energyerrorratio} shows the Hamiltonian
maximum error ratio computed for $h=0.04$ and $h=0.02$ for various
values of $\omega$ up to $4.5 \pi$. In the convergence region we would
expect that the ratio would be $16$ for the method of order $4$ and
$64$ for the method of order $6$, however the plot does not cover well
the convergence region. Overall, we see that the methods have a
conservation of $\mathcal{O}(h^2)$, except from the regions
corresponding to resonances. This behaviour seems to indicate that the
methods suffer of order reduction, a phenomenon that is not uncommon
for higher order methods in prescribed regions of the step-size. This
effect will be discussed more thoroughly below.

\subsection{Order reduction}
\label{sec:order-reduction}
Ultimately, it is the error in the slow variables one of the most
relevant quantities in the numerical simulations of these kind of
problems, because the fast variables will be in any case poorly resolved.
In figures (\ref{fig:FPUslowLGL4}--\ref{fig:FPUslowLGL6}) we show
the errors in the slow variables for the FPUT problem for different
values of the step-size and different $\omega$.
The errors are evaluated at $T=3$ and the exact solution is computed 
using Matlab's \texttt{ode45} to about machine precision
(setting \texttt{AbsTol, RelTol = 1e-14}). 
It is observed that the methods suffer of order reduction both in the positions and the
momenta, manifested as a platou in the error plots.

  \begin{figure}[t]
  \centering
  \includegraphics[width=.45\textwidth]{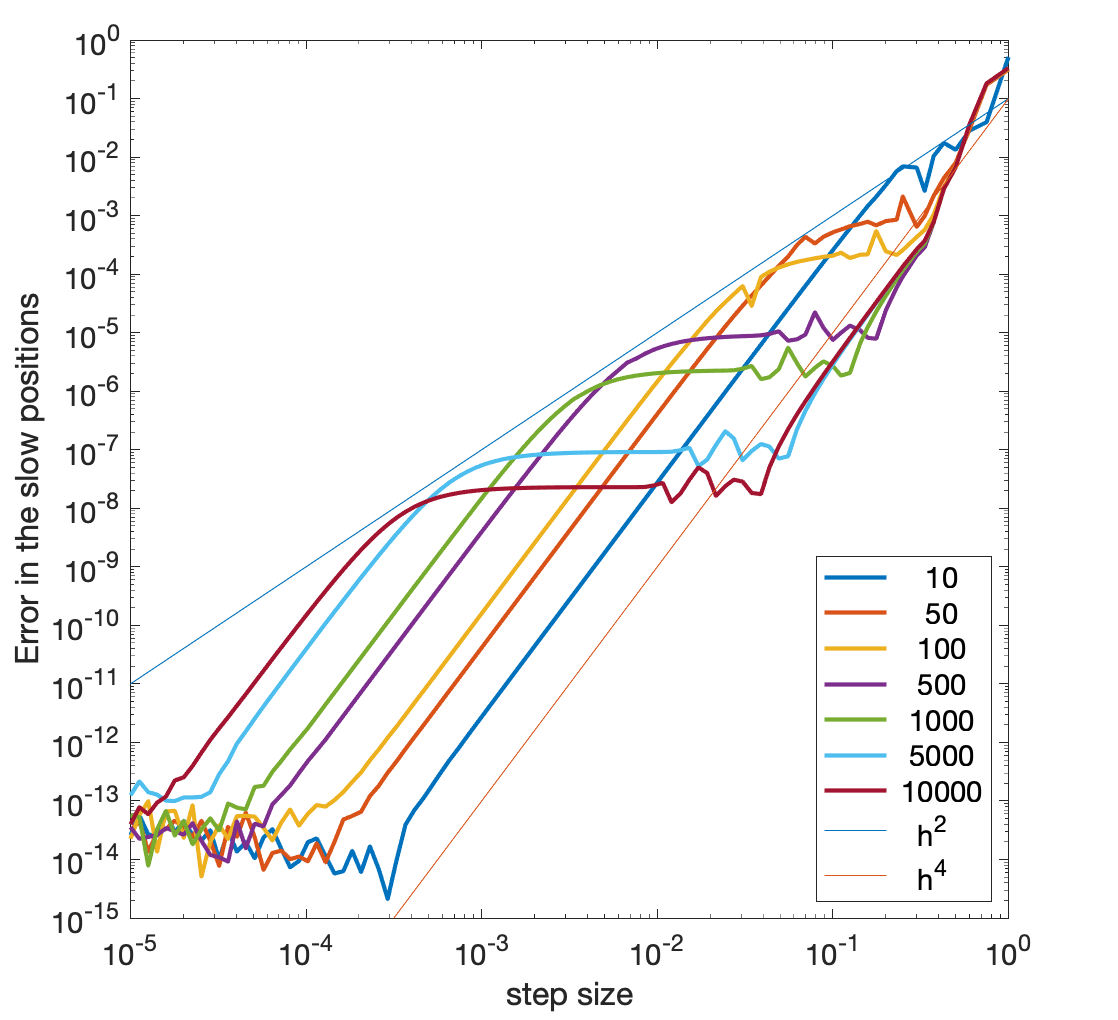}
    \includegraphics[width=.47\textwidth]{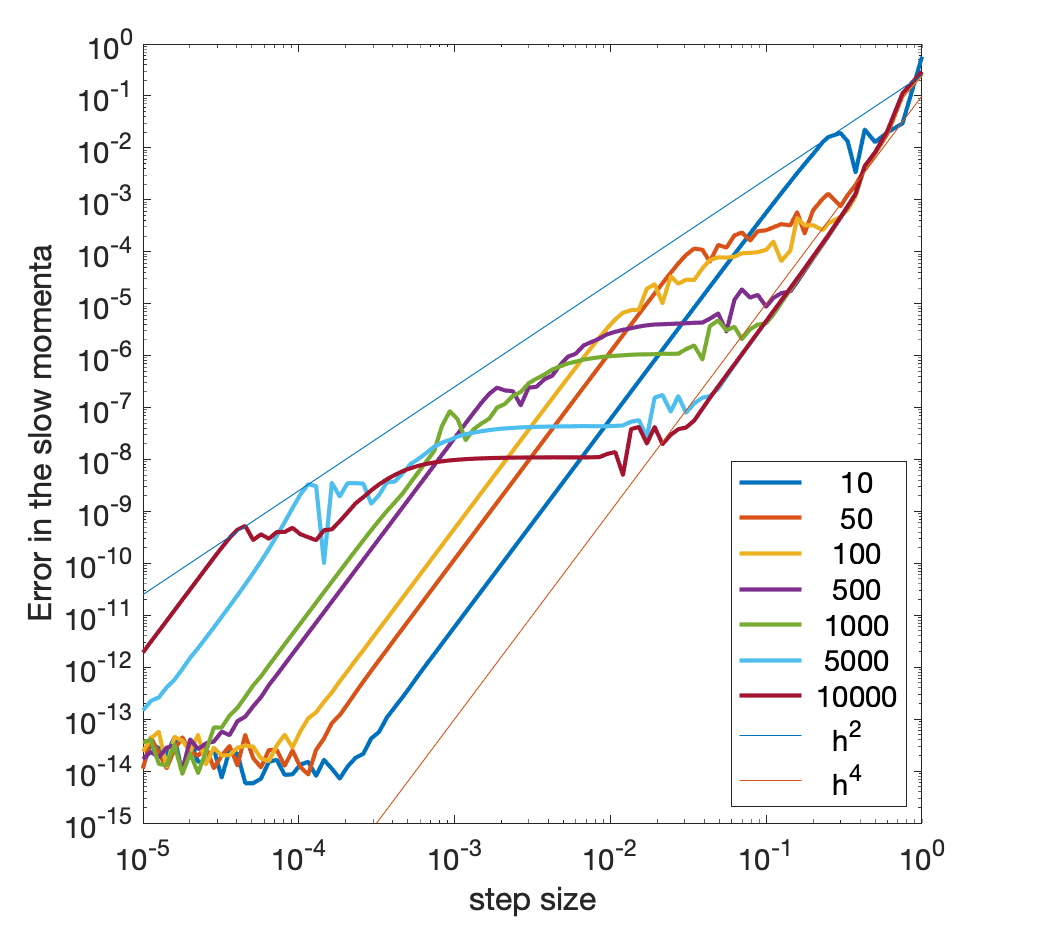}
  \caption{Errors at $T=3$ in the slow positions (left) and slow
    momenta (right) for the
    Lobatto-Gauss method of order 4 against the step size $h$ for
    $\omega = 10,\ldots, 10^4$. The lines for $h^2$ and $h^4$ are
    plotted for convenience.}
      \label{fig:FPUslowLGL4}
\end{figure}

\begin{figure}[t]
  \centering
  \includegraphics[width=.48\textwidth]{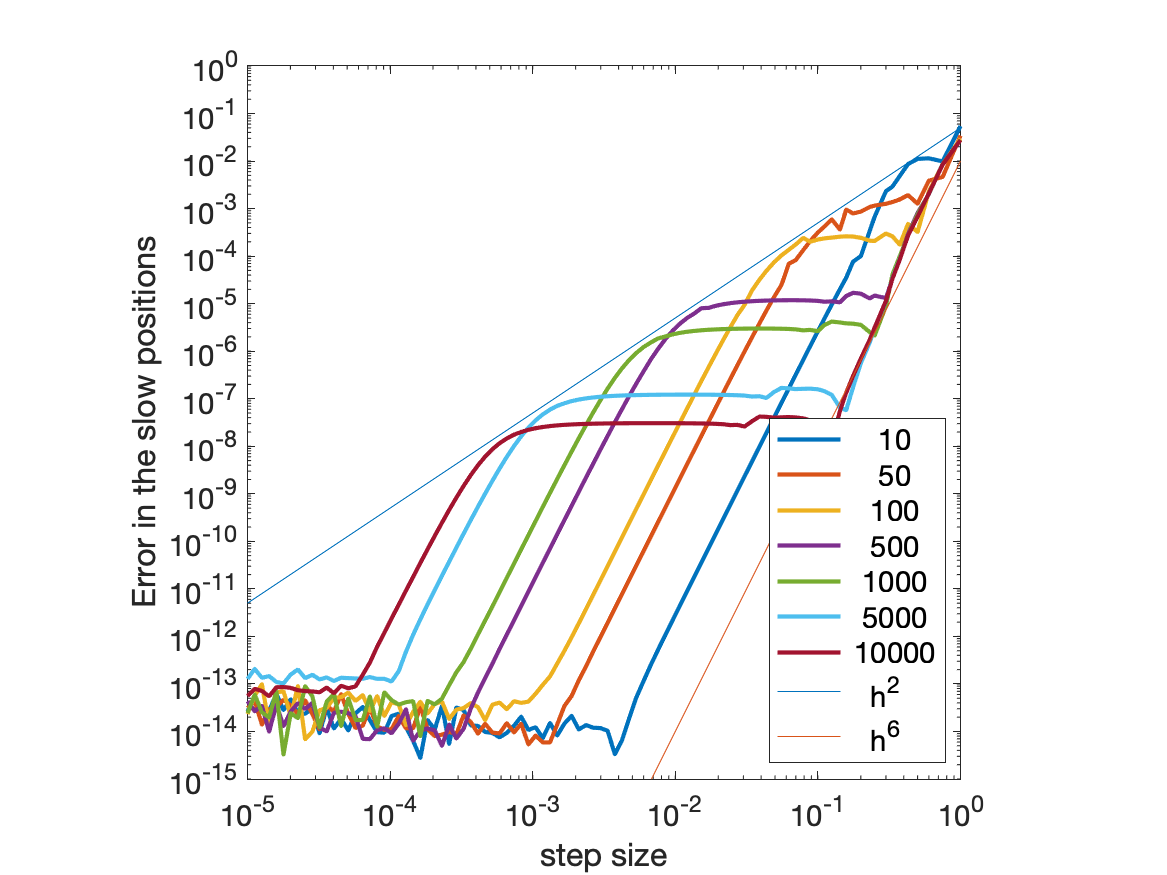}
    \includegraphics[width=.48\textwidth]{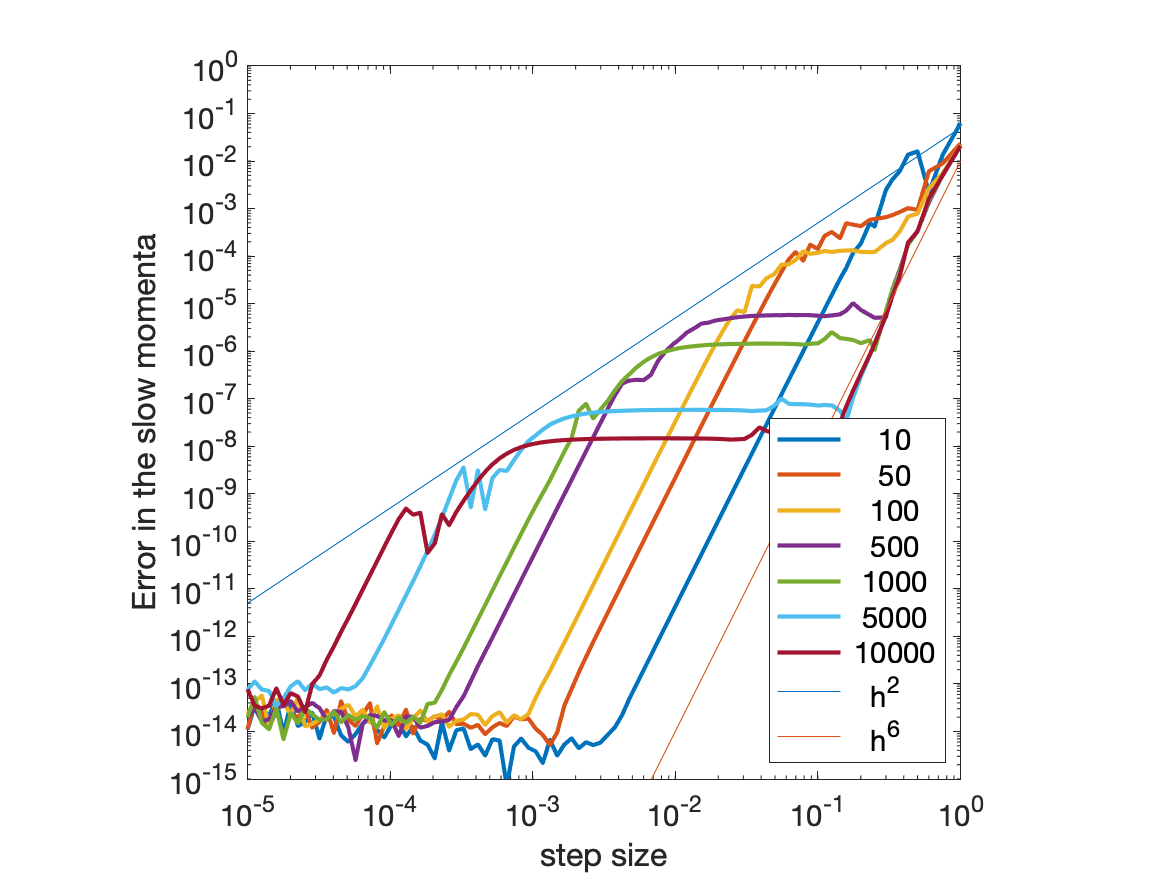}
  \caption{Error in the slow positions (left) and slow momenta (right)
    for the
    Lobatto-Gauss method of order 6.}
     \label{fig:FPUslowLGL6}
\end{figure}

In figures (\ref{fig:FPUslowIMEX4})-(\ref{fig:FPUslowIMEX6}) we 
repeat the same experiments by methods of order 4 and 6 obtained from
the IMEX and using the Yoshida technique \cite{yoshida90coh}. Also in this case one can
observe an order reduction, from order 4 to order 3 for the positions
and from order 4 to 2 for the momenta for the method of order 4. Similarly, one observes a reduction from order 6 to order 3 for
the positions and from order 6 to order 2 for the momenta for the
method of order 6. In summary, the order reduction is similar to that
of the Lobatto--Gauss-Legendre on the momenta, but is one order less
on the positions.

It is not clear why the Yoshida technique gives a lesser order
reduction for the positions and marginally also for the momenta. We
conjecture that it might be due to the fact that the method uses
step sizes $\alpha h$ and $\beta h$, rather than just $h$, and the use
of these two step sizes might reduce the resonance effects of the
single step size.
\begin{figure}[t]
  \centering
  \includegraphics[width=.48\textwidth]{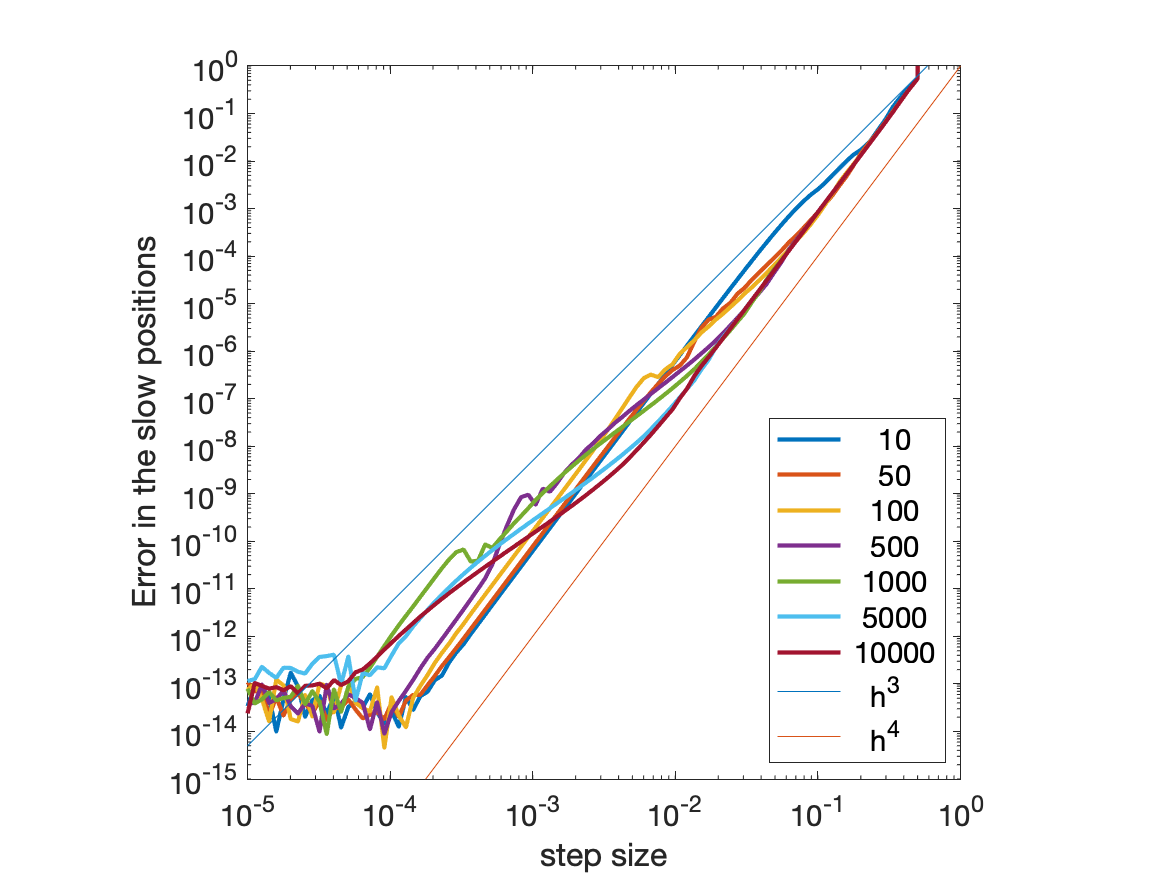}
    \includegraphics[width=.48\textwidth]{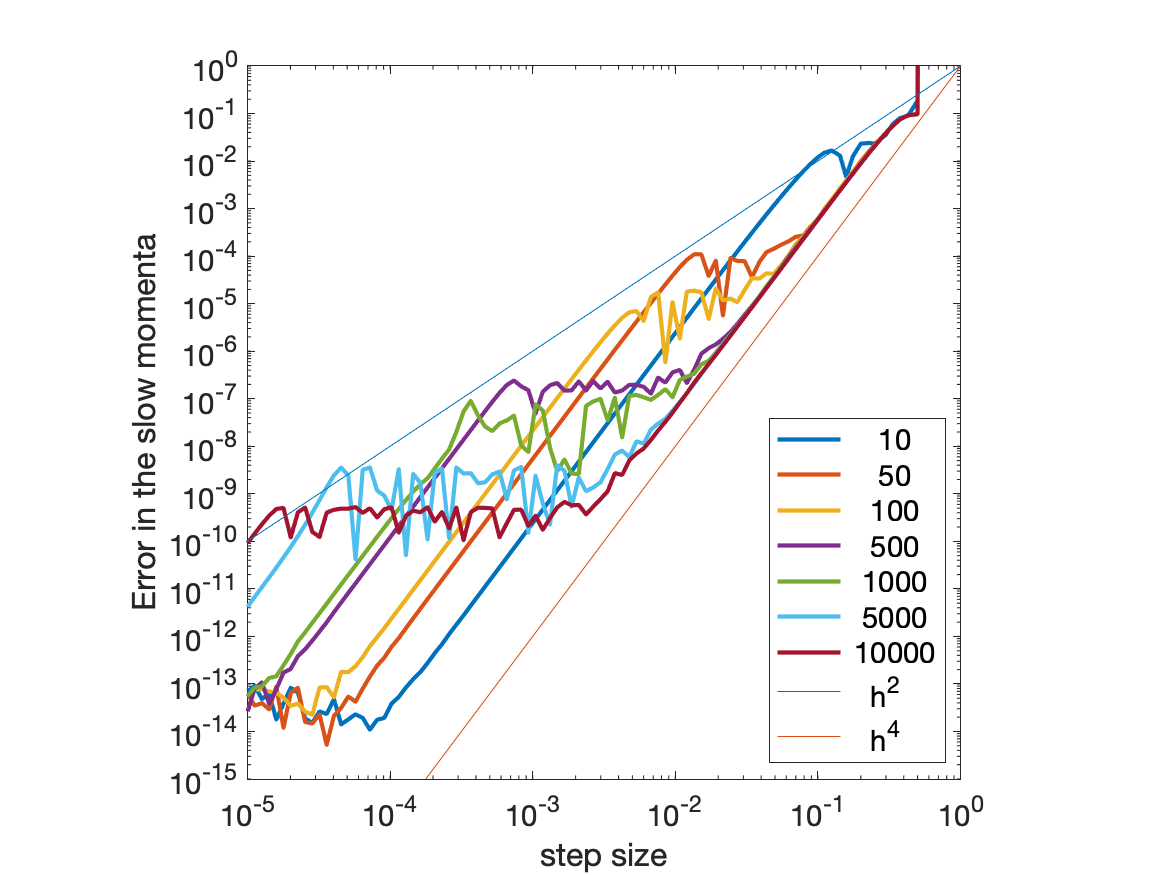}
  \caption{Error in the slow positions (left) and slow momenta (right)
    for the
    IMEX method with a Yoshida time stepping  for a method of order
    4.
    There is a order two reduction in the momenta, but only an order
    one reduction for the error in the slow positions.}
     \label{fig:FPUslowIMEX4}
\end{figure}

\begin{figure}[t]
  \centering
  \includegraphics[width=.48\textwidth]{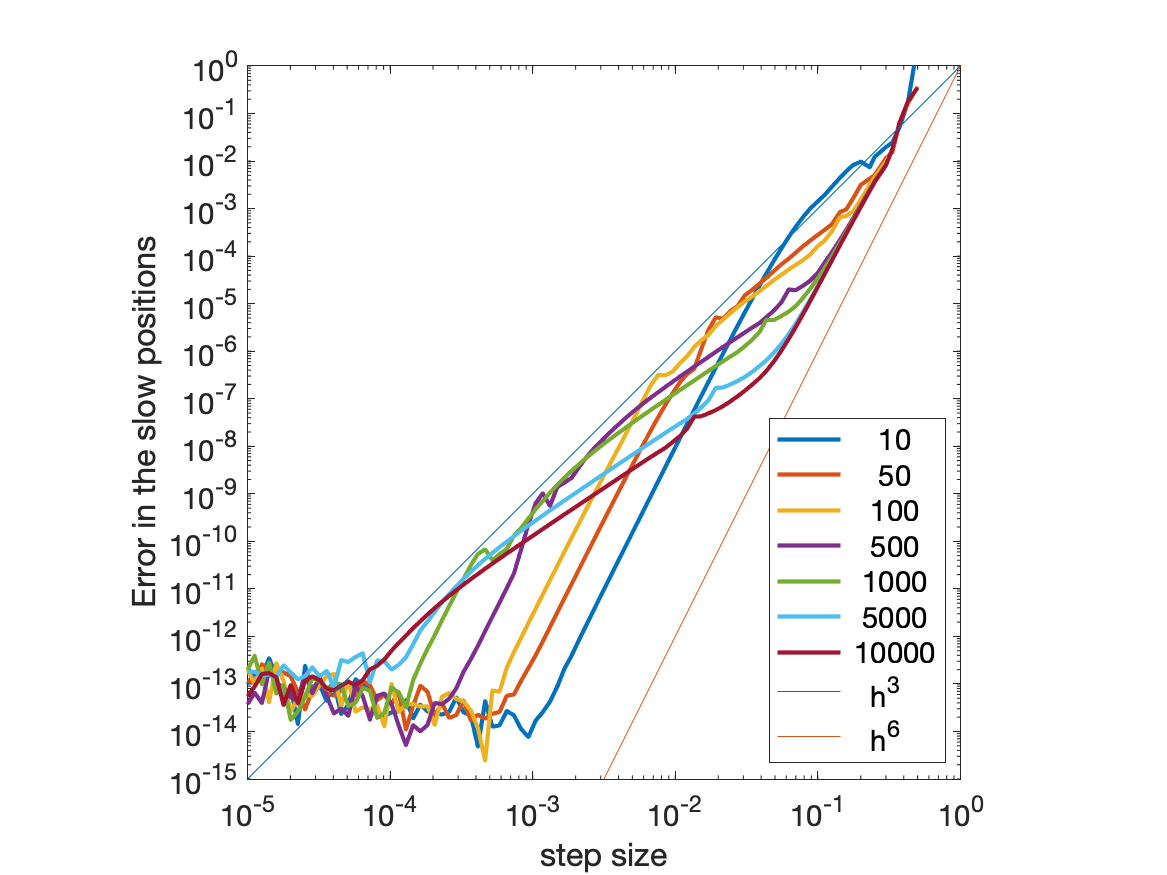}
    \includegraphics[width=.48\textwidth]{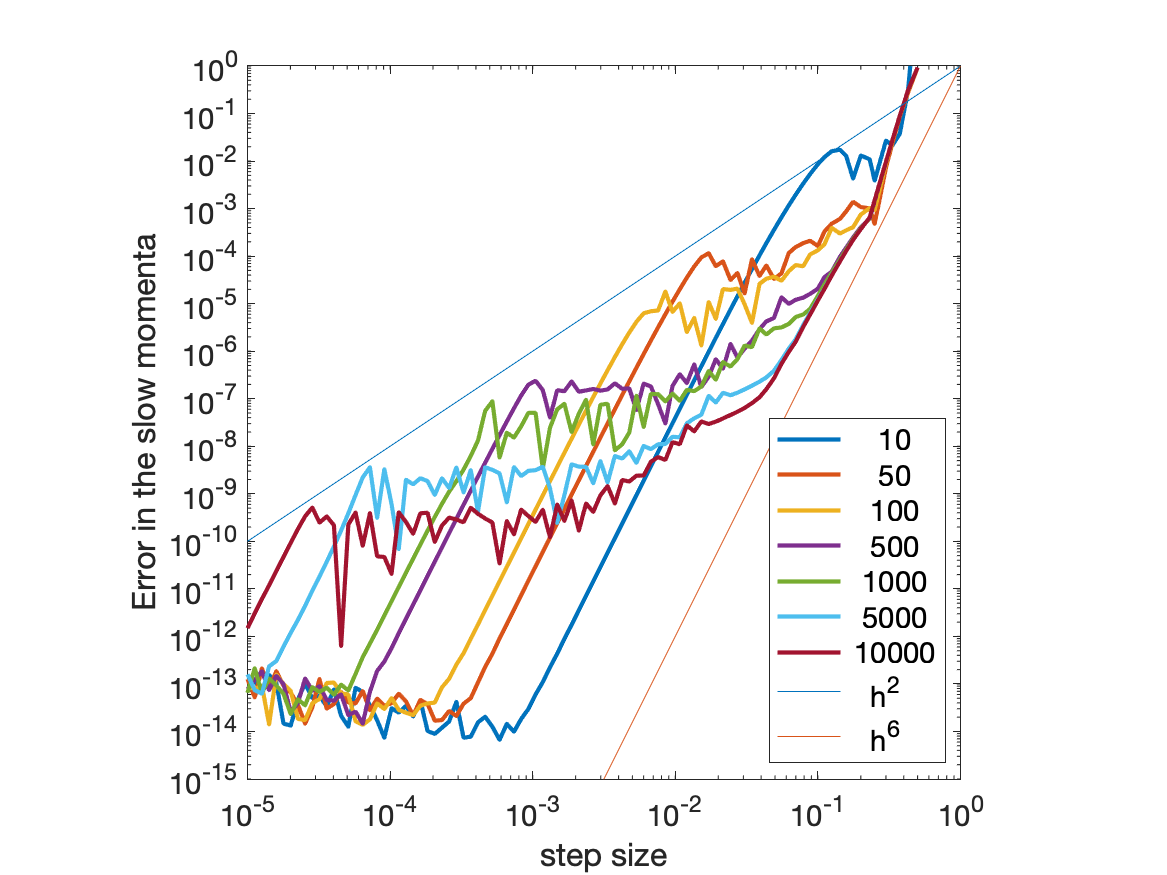}
  \caption{Error in the slow positions (left) and slow momenta (right) for the
    IMEX method with a Yoshida time stepping yielding a method of
    order 6. Also in this case there is an observable reduction in the
    order. We have four orders loss for the momenta and three order
    loss in the positions.}
    \label{fig:FPUslowIMEX6}
\end{figure}

Figures~(\ref{fig:FPUslowcomparison4})-(\ref{fig:FPUslowcomparison6})
show a comparison of the errors for methods of the same order. It is
observed that for larger step-sizes, the Lobatto--Gauss-Legendre have
smaller error (about two orders of magnitude) than IMEX with Yoshida
timestepping.  For smaller step-sizes, there is no obvious answer and
the choice of the method will most likely depend on the application
under consideration.

\begin{figure}[t]
  \centering
  \includegraphics[width=.48\textwidth]{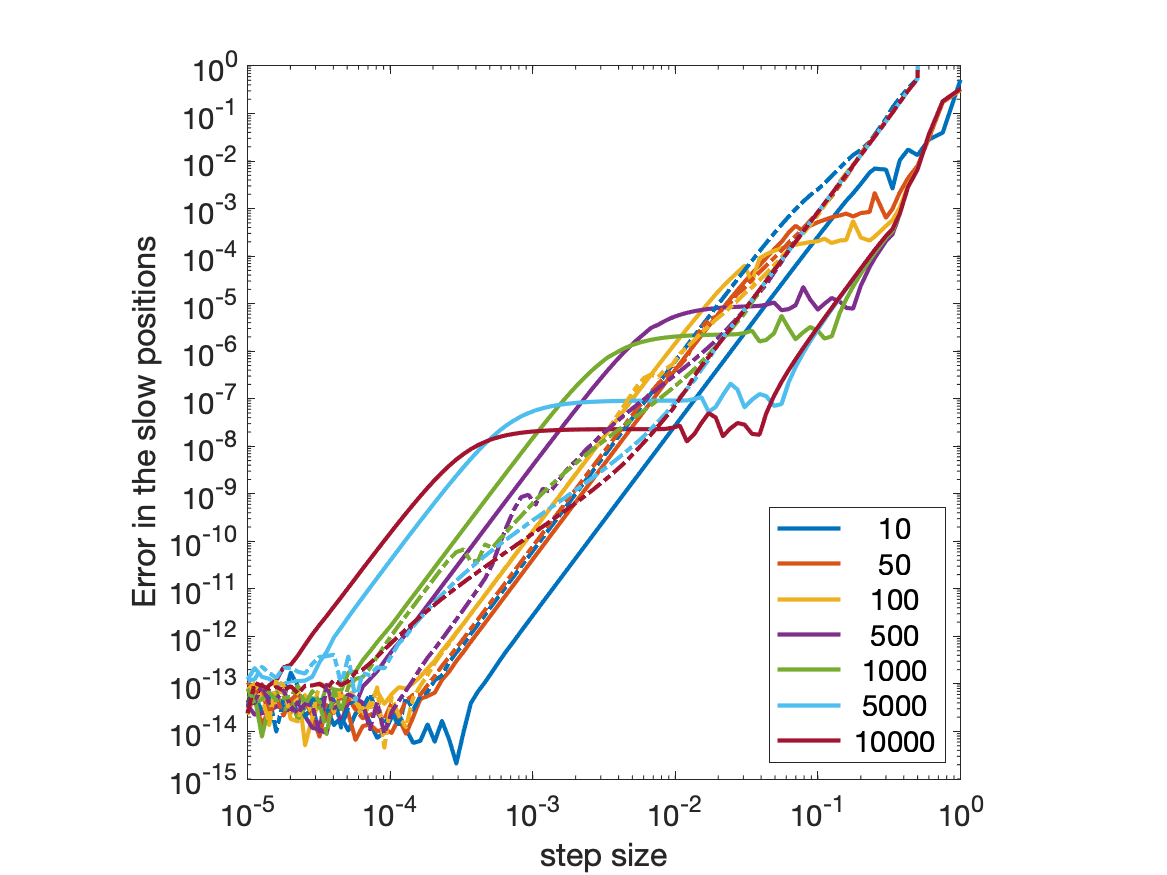}
  \includegraphics[width=.48\textwidth]{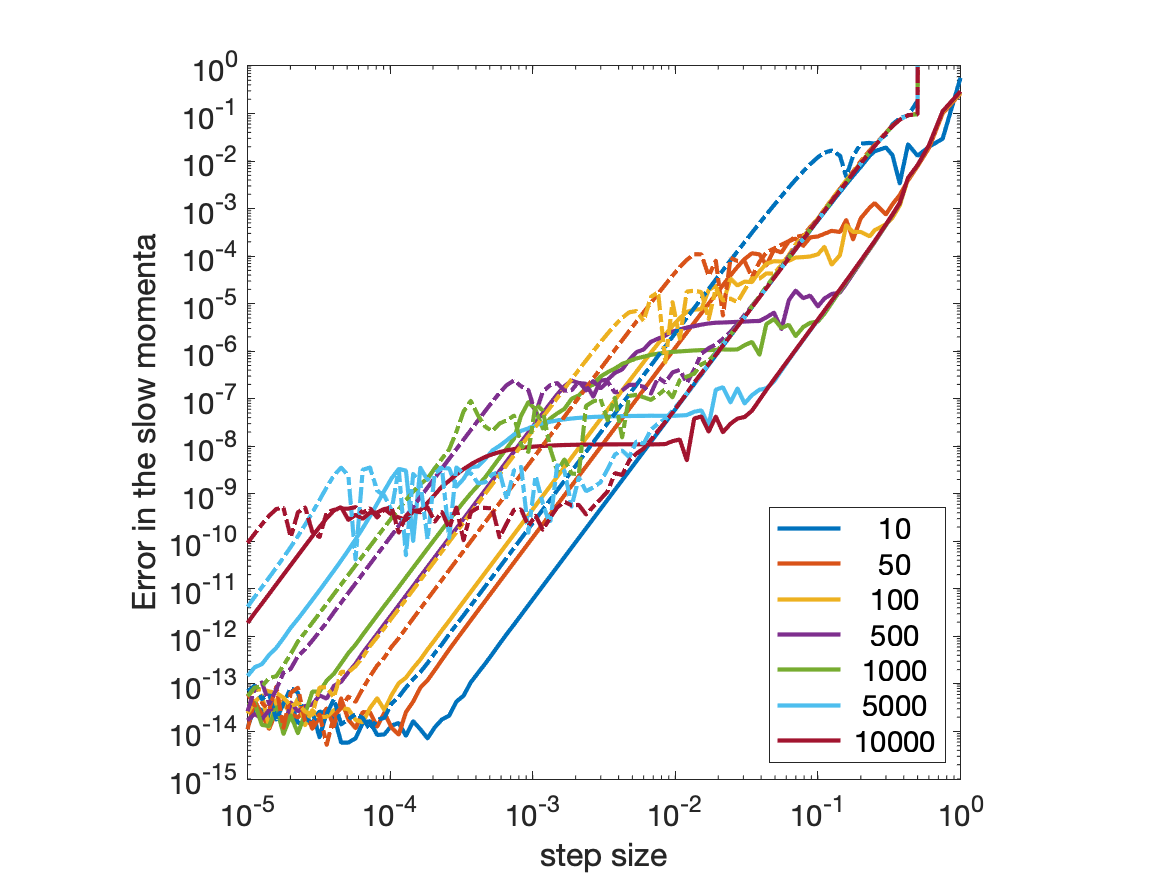}
  \caption{Comparison of the error in the slow positions (left) and
    slow momenta for the interpolation Lobatto--Gauss (solid line) and the IMEX-Yoshida
    method (dashed line) of order four.}
     \label{fig:FPUslowcomparison4}
\end{figure}

\begin{figure}[t]
  \centering
  \includegraphics[width=.48\textwidth]{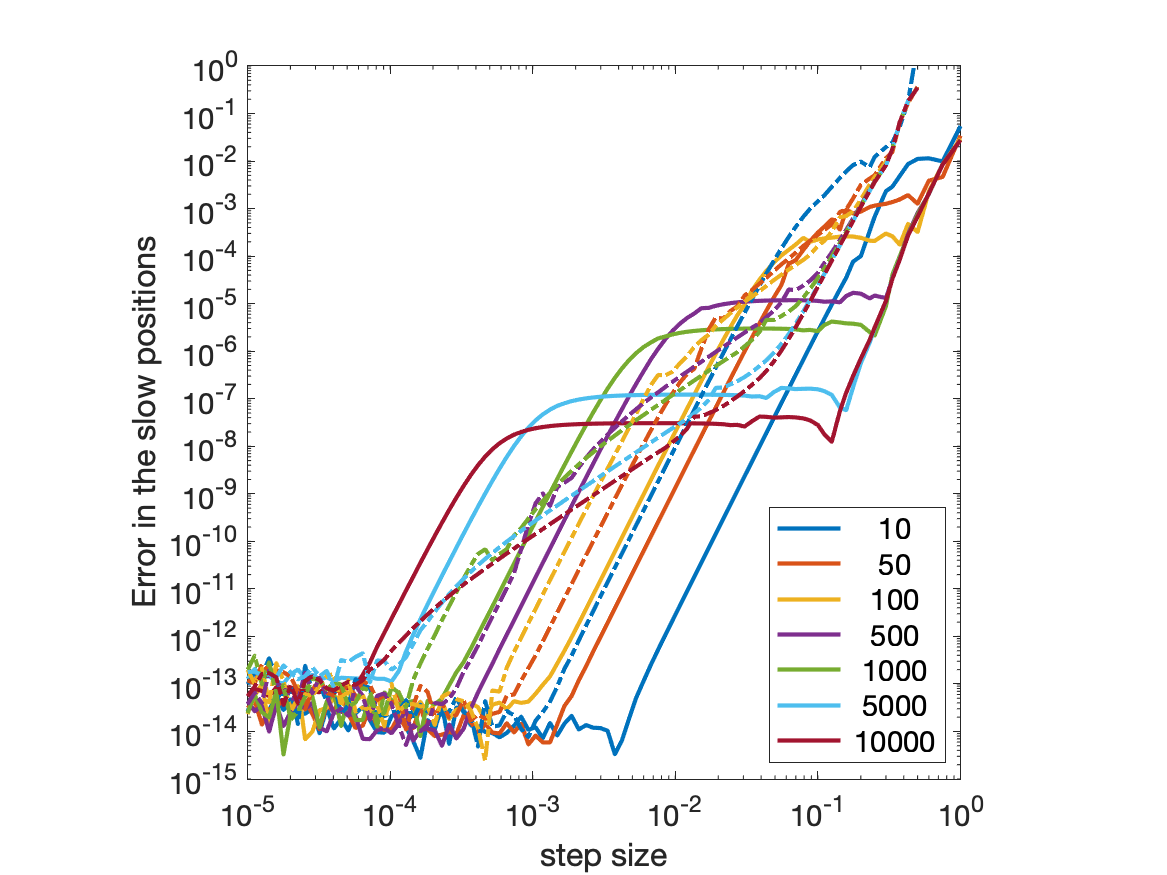}
  \includegraphics[width=.48\textwidth]{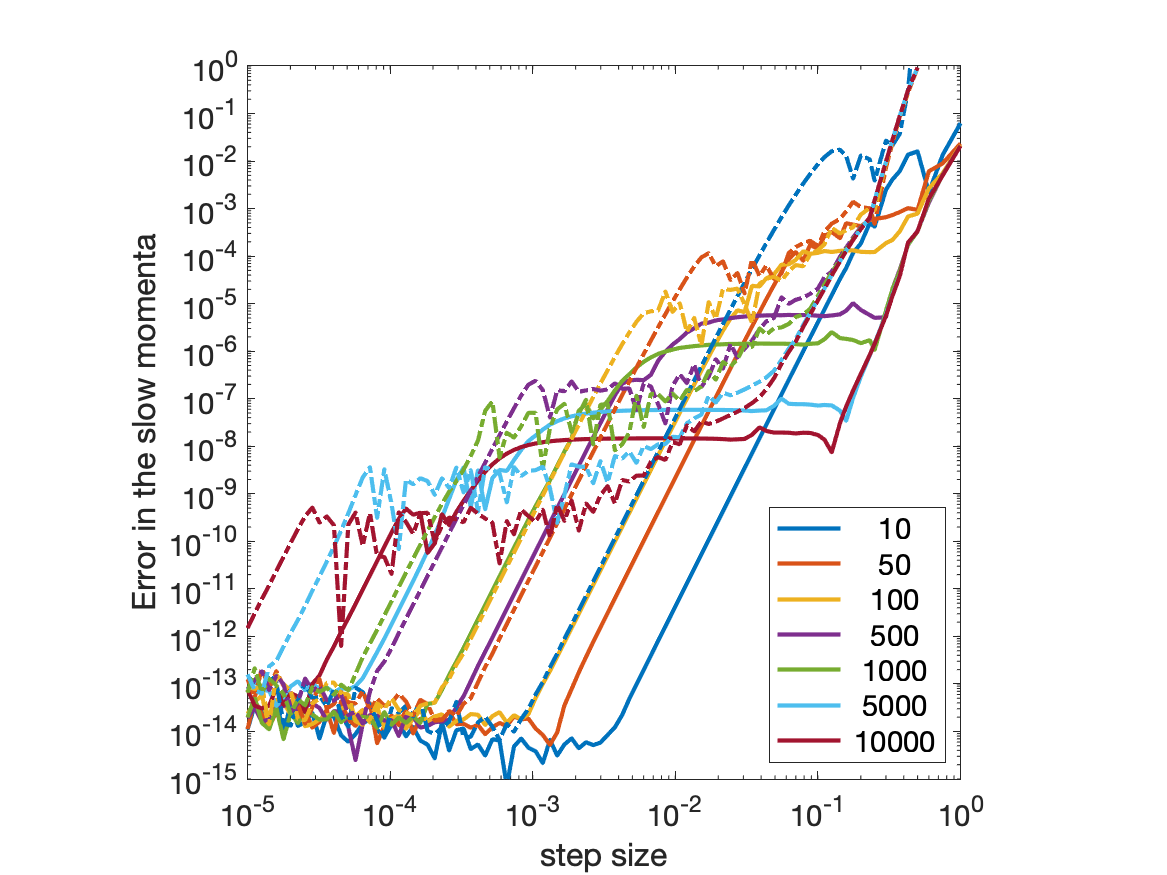}
  \caption{Comparison of the error in the slow positions (left) and
    slow momenta for the interpolation Lobatto--Gauss (solid line) and the IMEX-Yoshida
    method (dashed line) of order six.}
       \label{fig:FPUslowcomparison6}
     \end{figure}

The Lobatto--Gauss-Legendre of order 4 and 6 have
implicit stages, which require one and two functions evaluations
respectively. In our numerical experiments, we
have solved the implicit stages by fixed point iteration. The number
of function evaluations will then depend on the number of fixed point
iterations.
For small step sizes, we observed 1-2 fixed point iterations. For
larger step sizes (but still in the convergence region) we never
observed more than 10 iterations, a typical number was  $\approx$5-6. In comparison, the order 4 IMEX with the Yoshida technique would
require 3 function evaluations and 9 function evaluation for order
6. However, the Yoshida techniques have larger error in the regions of
convergence, especially in the larger step size regions. This error is
about two-three orders of magnitude larger than the Lobatto--Gauss-Legendre
methods, indicating that these can be used with a larger
step size, resulting in an overall cheaper method.

\section{Conclusions and further remarks}
\label{sec:concl-furth-remarks}
We have introduced a family of symplectic methods based on a
variational derivation. The main idea is to use different integration
quadrature formulas for different terms of the Lagrangian. The
introduction of extra internal stages is solved either by
interpolation or by collocation. In particular, we have derived a
higher order generalization of the IMEX method (using the Verlet
method and an interpolated form of the Implicit Midpoint Rule), namely
the LobattoIIIA-B--Gauss-Legendre family of arbitrary order, and
present the coefficients 
explicitly for the methods of order 4 and 6. We have proved that these
method possess the expected order   and shown that the
methods with internal stages solved by interpolation are
P-stable, making these particularly interesting in the context  of
oscillatory problem. We have also observed that these higher order 
methods might suffer from resonance and from order reduction.
The methods are thoroughly tested on the FPUT problem and their
behaviour is compared to higher order IMEX implementations using the
Yoshida time-stepping technique. 

The proposed methods might be considered as special subclass of additive
Runge--Kutta methods (ARK). The advantage of the variational derivation is
that the methods are automatically symplectic, therefore particularly
suited to geometric integration.
It will be interesting to explore further this mixed technique for
other choices of primary/secondary methods and the use other techniques,
like treating some of the terms by averaged Lagrangian methods in the spirit
of \cite{celledoni2017}. Possibly, this mixed approach might lead to
further interesting numerical method that might not be
easily discovered using the classical algebraic theory of RK and ARK
methods.

\section*{Acknowledgements}
\label{sec:acknowledgements}
The author would like to thank MSc Fredrick Pfeil for some very preliminary
results and simulations for the FPUT problem \cite{pfeil19aho}.
The final part of this work was completed at the Isaac Newton Institute for
Mathematical Sciences, which the author acknowledges for support and hospitality during the programme
\emph{Geometry, compatibility and structure preservation in
  computational differential equations} (2019), EPSRC grant number
EP/R014604/1.
This work was also partially supported by European Union Horizon 2020 research and innovation programme under the Marie Sklodowska-Curie grant agreement No. 691070, Challeges in preservation of structure (CHiPS).

\appendix

\section{The family of Lobatto IIIA-IIIB (primary) and
  Gauss-Legendre (secondary) methods }
\label{sec:family-lobatto-iiia}
\subsection{Methods based on interpolation}

\subsubsection{The IMEX}
\label{sec:imex-as-symplectic}
We consider the case the primary method for $L^1$ is the trapezoidal
rule, giving rise to the Verlet scheme, a Lobatto IIIA-IIIB pair PRK
with coefficients
$(A,b, c)$ and $(\widehat A, b,c)$
\begin{displaymath}
  \begin{array}{c|cc}
    0&0 &0 \\ 
    1 & \frac12 & \frac12\\[2pt] \hline \\[-12pt]
    & \frac12 & \frac12
  \end{array}, \qquad 
  \begin{array}{c|cc}
    0&\frac12 &0 \\
    1 & \frac12 & 0\\[2pt] \hline \\[-12pt]
    & \frac12 & \frac12
  \end{array}.
\end{displaymath}
The secondary scheme is the IMR ($s_2=1$, $ \tilde c_1 =
\frac12$, $\tilde b_1 = 1$).
One has
\begin{displaymath}
  \tilde A =
  \begin{bmatrix}
    \tilde a_{1,1} & \tilde a_{1,2} 
  \end{bmatrix}=
  \begin{bmatrix}
    \frac14 & \frac14
  \end{bmatrix}, \qquad
  \widehat {\tilde A}= 
\begin{bmatrix}
  \frac12\\[5pt] \frac12
\end{bmatrix}
\end{displaymath}

\subsubsection{Method of order four}
\label{sec:higher-order-methods}
To construct higher order methods  we look at the  Lobatto
IIIA-IIIB pair ($s=3$) and GL ($s=2$).
\begin{displaymath}
  A=\begin{bmatrix}
    0 &0& 0\\    \noalign{\medskip}
    \frac5{24} &\frac13 & -\frac1{24} \\     \noalign{\medskip}
    \frac16 & \frac23 & \frac16
  \end{bmatrix}
  \qquad
  \widehat A=\begin{bmatrix}
    \frac16 &-\frac16& 0\\     \noalign{\medskip}
    \frac16 &\frac13 & 0 \\     \noalign{\medskip}
    \frac16 & \frac56 & 0 
  \end{bmatrix}, 
\end{displaymath}
with
\begin{displaymath}
  c= \begin{bmatrix}0 &\frac12& 1\end{bmatrix}^T, \qquad
  b=\begin{bmatrix}\frac16& \frac23& \frac16\end{bmatrix}^T.
\end{displaymath}
For the Gauss-Legendre quadrature, we have
\begin{displaymath}
  \tilde c= \begin{bmatrix}\frac12-\frac{\sqrt3}6& \frac12+\frac{\sqrt3}6 \end{bmatrix}^T,
  \qquad \tilde b=\begin{bmatrix}\frac12 & \frac12\end{bmatrix}^T.
\end{displaymath}
We consider the interpolation case \R{eq:18}. The matrix $\tilde A$ and $\widehat {\tilde A}$ are 
\begin{displaymath}
  \tilde A =
  \begin{bmatrix}
    \frac16-\frac{\sqrt{3}}{36}& \frac13-\frac{\sqrt{3}}9&
    -\frac{\sqrt{3}}{36}\\     \noalign{\medskip}
    \frac16+\frac{\sqrt{3}}{36}& \frac13+\frac{\sqrt{3}}{9}&\frac{\sqrt{3}}{36}
  \end{bmatrix}, \qquad
  \widehat {\tilde A}=
   \begin{bmatrix}
    \frac{\sqrt{3}}{12}&-\frac{\sqrt{3}}{12}\\     \noalign{\medskip}
    \frac14+\frac{\sqrt{3}}{12}&\frac14-\frac{\sqrt{3}}{12}\\     \noalign{\medskip}
     \frac12+\frac{\sqrt{3}}{12} & \frac12-\frac{\sqrt{3}}{12}
   \end{bmatrix}
 \end{displaymath}
As for the primary method, we have $Q_1 = q_0$ and $Q_3= q_1$.

\subsubsection{Method of order six}
Consider the  Lobatto IIIA-IIIB pair ($s=4$) and GL ($s=3$),
\begin{displaymath}
  A =
  \begin{bmatrix}
    0& 0 & 0 & 0\\
    \frac{11+\sqrt{5}}{120} &  \frac{25-\sqrt{5}}{120} &
    \frac{25-13\sqrt{5}}{120} &  \frac{-1+\sqrt{5}}{120}\\
    \noalign{\medskip}
     \frac{11-\sqrt{5}}{120} &  \frac{25+13\sqrt{5}}{120} &
     \frac{25+\sqrt{5}}{120} &  \frac{-1-\sqrt{5}}{120}\\
         \noalign{\medskip}
     \frac1{12} & \frac5{12} & \frac5{12} & \frac1{12}
  \end{bmatrix},
  \qquad
  \widehat A =
  \begin{bmatrix}
    \frac1{12} & \frac{-1-\sqrt{5}}{24} & \frac{-1+\sqrt{5}}{24} & 0\\
        \noalign{\medskip}
    \frac1{12} & \frac{25+\sqrt{5}}{120} & \frac{25-13\sqrt{5}}{120} &
    0\\
        \noalign{\medskip}
    \frac1{12} & \frac{25+13\sqrt{5}}{120} & \frac{25-\sqrt{5}}{120} &
    0\\
        \noalign{\medskip}
    \frac1{12} & \frac{11-\sqrt{5}}{24} & \frac{11+\sqrt{5}}{24} & 0\\
  \end{bmatrix}
\end{displaymath}
with
\begin{displaymath}
  c=  \begin{bmatrix}
    0 & \frac12-\frac{\sqrt{5}}{10} & \frac12+\frac{\sqrt{5}}{10} & 1
  \end{bmatrix}^T, \qquad
  b=
  \begin{bmatrix}
    \frac1{12} & \frac5{12} & \frac5{12} & \frac1{12}
  \end{bmatrix}^T.
\end{displaymath}
For the Gauss-Legendre quadrature, we have
\begin{displaymath}
  \tilde c= \begin{bmatrix}\frac12-\frac{\sqrt{15}}{10}& \frac12 & \frac12+\frac{\sqrt{15}}{10} \end{bmatrix}^T,
  \qquad \tilde b=\begin{bmatrix}\frac5{18} & \frac49&\frac5{18} \end{bmatrix}^T.
\end{displaymath}
We consider the interpolation case \R{eq:18}. The matrix $\tilde A$ and $\widehat {\tilde A}$ are 
\begin{displaymath}
  \tilde A= \begin {bmatrix} \frac1{15}&{\frac { 25 -6
\sqrt {15}+3 \sqrt {5}}{120}}&{\frac { 25-6
\sqrt {15}-3\sqrt {5}}{120}}&{\frac{1}{60}}
\\
\noalign{\medskip}{\frac{5}{48}}&{\frac{5}{24}}+\frac{\sqrt{5}}{16}&{
  \frac{5}{24}}-\frac{\sqrt5}{16}&-\frac1{48}\\
\noalign{\medskip}
\frac1{15}&{\frac { 25+6\sqrt {15}+3\sqrt {5}}{120}}
&{\frac {25 + 6\sqrt {15}-3\sqrt {5}}{120}}&{\frac{
1}{60}}\end {bmatrix} 
\qquad
\widehat{\tilde A} =
\begin{bmatrix}
  \frac1{18}&-\frac19&\frac1{18}\\
  \noalign{\medskip}{\frac { 25+ 6\sqrt {15}-3 \sqrt {5}}{180}}&\frac29-\frac{\sqrt{5}}
{15}&{\frac { 25 -6\sqrt {15}-3\sqrt {5}}{180}}\\
\noalign{\medskip}{\frac {25+  6\sqrt {15}+3 \sqrt
    {5}}{180}}&\frac29+\frac{\sqrt5}{15}&
{\frac { 25 -6\sqrt {15}+3 \sqrt {5}}{180}}
\\ \noalign{\medskip}\frac29&\frac59&\frac29

\end{bmatrix}
\end{displaymath}

\subsection{Methods based on collocation}
The weights $b, \tilde b$ and nodes $c, \tilde c$ of the primary and
secondary method of each order, as well as the correspoding PRK for
the primary methods are the same as for interpolation.
The difference is in the coefficient matrices $\tilde A$ and $\widehat{\tilde A}$, which we report
below for convenience.
\subsubsection{Second order method}
\begin{displaymath}
  \tilde A =
  \begin{bmatrix}
    \frac38 & \frac18
  \end{bmatrix}, \qquad
  \widehat{\tilde A} =
  \begin{bmatrix}
    \frac14 \\[5pt] \frac34
  \end{bmatrix}.
\end{displaymath}

\subsubsection{Fourth order method}
\begin{displaymath}
  \tilde A =
  \begin{bmatrix}
    \frac16-{\frac {\sqrt {3}}{108}}& \frac13-{\frac {4
        \sqrt {3}}{27}}&-{\frac {\sqrt {3}}{108}}\\
    \noalign{\medskip}\frac16+{
\frac {\sqrt {3}}{108}}&\frac13+{\frac {4\sqrt {3}}{27}}&{\frac {\sqrt {
3}}{108}}
  \end{bmatrix}, \qquad
  \widehat{\tilde A} =
  \begin{bmatrix}
    \frac{\sqrt3}{36}&-\frac{\sqrt {3}}{36}
\\ \noalign{\medskip}\frac14+\frac{\sqrt {3}}9&\frac14-\frac{\sqrt {3}}9
\\ \noalign{\medskip}\frac12+\frac{\sqrt {3}}{36}&\frac12-\frac{\sqrt {3}}{36}
  \end{bmatrix}.
\end{displaymath}

\subsubsection{Six order method}
\begin{displaymath}
  \tilde A = 
  \begin{bmatrix}
    {\frac{19}{240}}&{\frac {\sqrt {5}
 \left( \sqrt {15}-5 \right) ^{2} \left( 3\,\sqrt {15}+4\,\sqrt {5}+2
\,\sqrt {3}+12 \right) }{2400}}&-{\frac {\sqrt {5} \left( \sqrt {15}-5
 \right) ^{2} \left( 3\,\sqrt {15}-2\,\sqrt {3}-4\,\sqrt {5}+12
 \right) }{2400}}&{\frac{1}{240}}\\ \noalign{\medskip}{\frac{17}{192}}
&{\frac{5}{24}}+{\frac {5\,\sqrt {5}}{64}}&{\frac{5}{24}}-{\frac {5\,
\sqrt {5}}{64}}&-{\frac{1}{192}}\\ \noalign{\medskip}{\frac{19}{240}}&
-{\frac {\sqrt {5} \left( \sqrt {15}+5 \right) ^{2} \left( 3\,\sqrt {
15}-4\,\sqrt {5}+2\,\sqrt {3}-12 \right) }{2400}}&{\frac {\sqrt {5}
 \left( \sqrt {15}+5 \right) ^{2} \left( 3\,\sqrt {15}+4\,\sqrt {5}-2
\,\sqrt {3}-12 \right) }{2400}}&{\frac{1}{240}}
\end{bmatrix},
\end{displaymath}
\begin{displaymath}
  \widehat{\tilde A} =
  \begin{bmatrix}
    {\frac{1}{72}}&-\frac1{36}&{\frac{1}{72}}
\\ \noalign{\medskip}{\frac{5}{36}}+{\frac { \left( 12\,\sqrt {3}-3
 \right) \sqrt {5}}{360}}&\frac29-\frac{\sqrt5}{12}&{\frac{5}{36}}+{\frac {
 \left( -12\,\sqrt {3}-3 \right) \sqrt {5}}{360}}\\ \noalign{\medskip}
{\frac{5}{36}}+{\frac { \left( 12\,\sqrt {3}+3 \right) \sqrt {5}}{360}
}&\frac29+\frac{\sqrt5}{12}&{\frac{5}{36}}+{\frac { \left( -12\,\sqrt {3}+3
 \right) \sqrt {5}}{360}}\\ \noalign{\medskip}{\frac{19}{72}}&{\frac{
17}{36}}&{\frac{19}{72}}

  \end{bmatrix}.
\end{displaymath}


\end{document}